\documentclass[reqno,oneside,10pt]{amsart}

\usepackage[T1]{fontenc}
\usepackage{times,mathptm}
\usepackage{amssymb,epsfig,verbatim,mathtools,float,hyperref,xcolor,amsmath}
\usepackage[mathscr]{eucal}
\usepackage[all,cmtip]{xy}
\theoremstyle{plain}

\newtheorem{thm}{Theorem}[section]
\newtheorem{cor}[thm]{Corollary}
\newtheorem{prob}[thm]{Problem}
\newtheorem{pro}[thm]{Proposition}
\newtheorem{lem}[thm]{Lemma}
\newtheorem{conjecture}[thm]{Conjecture}
\newtheorem{proposition-principale}[thm]{Proposition principale}
\newtheorem{theoalph}{Theorem}
\newtheorem{claim}{Claim}

\theoremstyle{definition}
\newtheorem{defi}[thm]{Definition}
\newtheorem{que}{Question}[section]
\newtheorem{eg}[thm]{Example}
\newtheorem{rem}[thm]{Remark}
\newtheorem{rems}[thm]{Remarks}

\addtocounter{section}{0}              
\numberwithin{equation}{section}       
\usepackage[left,pagewise]{lineno}

\begin{document}

\setlength{\baselineskip}{0.54cm}         
\title{Holomorphic vector fields with a barycentric condition}
\date{}

\author{Dominique Cerveau}
\address{Univ. Rennes, CNRS, IRMAR-UMR $6625$, F-35000 Rennes, France}
\email{dominique.cerveau@univ-rennes1.fr}

\author{Julie D\'eserti}
\address{Universit\'e d'Orl\'eans, Institut Denis Poisson, route de Chartres, $45067$ Orl\'eans Cedex $2$, France}
\email{deserti@math.cnrs.fr}

\author{Alcides Lins Neto}
\address{IMPA, Estrada Dona Castorina, 110, Horto, Rio de Janeiro, Brasil}
\email{alcides@impa.br}



\begin{abstract} 
We study the $p$-tuples  of holomorphic vector fields $(X_1,X_2,\ldots,X_p)$
satisfying the barycentric property $\displaystyle\sum_k\exp tX_k=p\cdot\mathrm{id}$, where 
$\exp tX$ denotes the flow of $X$.
\end{abstract}

\maketitle

\section{Introduction}

Let $\mathcal{U}$ be a connected open subset of $\mathbb{R}^n$
(resp. $\mathbb{C}^n$).
Let $X_1$, $X_2$, $\ldots$, 
$X_p$ be $p$ analytic (resp. holomorphic) distinct vector 
fields on $\mathcal{U}$. Denote by 
$\varphi_t^k=\exp(tX_k)$ the local one-parameter
subgroup of~$X_k$; it is the solution of the following 
ordinary differential equation
\[
\frac{d\varphi_t^k(x)}{dt}=X_k(\varphi_t^k(x))
\]
with initial data $\varphi_0^k(x)=x$.

For any point $x\in\mathcal{U}$, $\varphi^k_t(x)$ 
is well-defined for $t$ sufficiently small and 
we assume that 
\begin{equation}\label{eq:chambar}
\displaystyle\sum_{k=1}^pX_k(\varphi_t^k(x))=0\qquad\forall\,x\in\mathcal{U}\text{ and $t$ small.}
\end{equation}
In particular $\displaystyle\sum_{k=1}^p\exp(tX_k)=p\mathrm{id}$ and
by doing $t=0$ in  (\ref{eq:chambar}) we get 
\[
\displaystyle\sum_{k=1}^pX_k=0.
\]

Let us give an interpretation of $(\ref{eq:chambar})$: 
at any point $x$ there are $p$ identical particles 
transported by the vector fields $X_k$ while 
preserving their barycenter at the initial position
$x$. The condition $(\ref{eq:chambar})$ is called 
\textsl{barycentric property}. A set of $p$ 
vector fields 
$X_1$, $X_2$, $\ldots$, $X_p$ satisfying 
the barycentric property is called a \textsl{$p$-chambar} 
and is denoted $\mathrm{Ch}(X_1,X_2,\ldots,X_p)$. 
In \S \ref{sec:remeg} we give a long list of detailed examples. 
The barycentric property produces 
interesting ordinary differential equations in 
dimension $\geq 1$.

\begin{rem}
The barycentric property is invariant by 
affine transformations. Let 
$\mathrm{Ch}(X_1,X_2,\ldots,X_p)$ be a $p$-chambar
in some open subset $\mathcal{U}\subset\mathbb{C}^n$
and let $T$ be an affine transformation of 
$\mathbb{C}^n$. Then the vector fields 
$T_*X_1$, $T_*X_2$, $\ldots$, $T_*X_p$ satisfy
the barycentric condition. 

In fact, if a biholomorphism 
$f\colon \mathcal{U}\to f(\mathcal{U})\subset\mathbb{C}^n$
sends any set of vector fields on 
$\mathcal{U}$ with the barycentric
property into another set with the 
barycentric property, then $f$ is 
an affine transformation. However, 
in some particular cases of $p$-chambars 
there are other types of biholomorphisms
with this property (see for instance
Theorem \ref{thm:cstpchamb}).
\end{rem}

If $X$ is a vector field on $\mathcal{U}$, then 
$\mathcal{F}_X$ denotes the foliation (maybe 
singular) whose leaves are the integral curves 
of $X$. Hence $\mathcal{F}_X$ is a foliation by 
(real or complex) curves. From now on all the 
vector fields $X_k$ are not identically
zero.

In the case of a $2$-chambar $\mathrm{Ch}(X_1,X_2)$ 
condition (\ref{eq:chambar}) implies that 
$\mathcal{F}_{X_1}=\mathcal{F}_{X_2}$. We have also (Theorem~\ref{thm:dim1}):

\begin{theoalph}\label{thm:2chambar}
{\sl Let $\mathcal{U}$ be an open subset of $\mathbb{R}^n$ 
$($resp. $\mathbb{C}^n)$.
Let $X_1$, $X_2$ be two analytic $($resp. holomorphic$)$ vector fields 
on $\mathcal{U}$. Assume that $X_1$ and $X_2$ satisfy the barycentric 
property. 

Then 
$\mathcal{F}_{X_1}=\mathcal{F}_{X_2}$, and it is a 
foliation by straight lines:
\begin{itemize}
\item[$\diamond$] the closure of the generic leaves
are intersection of lines with the open subset 
$\mathcal{U}$;

\item[$\diamond$] on each line the flow 
$\varphi_t^k=\exp(tX_k)$, $k=1$, $2$, coincides with the 
flow of a constant vector field.
\end{itemize}}
\end{theoalph}

The link between the $2$-chambards and 
the foliations by straight lines suggests that certain special 
dynamics appear in dimension strictly larger than $1$.

In \S \ref{sec:remeg} we will construct explicit examples 
satisfying Theorem \ref{thm:2chambar}. It is 
sufficient to consider any foliation by straight lines
$\mathcal{F}$ (maybe singular) and to take a vector 
field $X$ whose restriction to each leaf is "constant". 

\smallskip

In the algebraic case the foliations by 
straight lines are classified on $\mathbb{P}^2_\mathbb{C}$
and~$\mathbb{P}^3_\mathbb{C}$. We will see that 
in this case the flows associated to a global 
algebraic $2$-chambar are some special birational
flows (\S\ref{sec:2cham}).

\smallskip

We will consider the case of colinear vector 
fields (a condition satisfied by the $2$-chambars),
{\it i.e.} the case where $X_i=a_iX$ with $a_i$ 
constant for any $1\leq i\leq p$; such chambars are called rigid chambars. The barycentric
property implies that $\mathcal{F}_X$ is a 
foliation by straight lines in the real case (Theorem~\ref{thm:realnotcomp}) but not in 
the complex case. We
will see the two following results (\S\ref{sec:rigidcham}, Theorem \ref{thm:rigidp} and Corollary \ref{cor:rigidp}):

\begin{theoalph}
{\sl If $\mathrm{Ch}(a_1X,a_2X,\ldots,a_pX)$, 
$a_k\in\mathbb{C}^*$, is a rigid $p$-chambar on the connected open 
set $\mathcal{U}\subset\mathbb{C}^n$, 
then the flow $\exp tX$ of $X$ is polynomial 
of degree at most $p-1$ as a function of the time $t$.
In particular, the orbits of~$X$ are contained in some 
rational curves.}
\end{theoalph}

\begin{theoalph}
{\sl Let $\mathrm{Ch}(a_1X,a_2X,\ldots,a_pX)$ be a rigid $p$-chambar 
on an open set $\mathcal{U}\subset\mathbb{C}^n$. 
If $X$ has a singular point, then the
set $\mathrm{Sing}(X)$ of $X$ has
dimension $\geq 1$.}
\end{theoalph}

\smallskip

We will see also examples where the $X_i$'s are 
polynomial vector fields, and more generally 
rational vector fields. In particular, in the 
linear case we get (\S\ref{sec:linnil}, Theorem \ref{thm:linnil}):

\begin{theoalph}
{\sl Let $X_1$, $X_2$, $\ldots$, $X_p$ be some linear 
vector fields on $\mathcal{U}\subset\mathbb{R}^n$ 
$($resp. $\mathbb{C}^n)$. 

If they satisfy the barycentric property ,then they are nilpotent.
In particular, the flows $\exp(tX_k)$ are
polynomials in $t$.}
\end{theoalph}

In the case of $3$-chambars one gets (Theorem \ref{thm:lin3cham}):

\begin{theoalph}
{\sl Let $X_1$, $X_2$, $X_3$ be some linear vector fields
on $\mathbb{C}^n$. 

If they satisfy the barycentric property, then, up to 
conjugacy, they are contained in the 
Heisenberg Lie algebra $\mathfrak{h}_n$ 
$($we identify $X_i$ with its matrix$)$.}
\end{theoalph}

We then give the classification of the 
$3$-chambars in dimension $1$, all
chambars appearing in this classification are rigid (\S\ref{sec:3cham4cham}, Theorem \ref{thm:loc3bfc}):

\begin{theoalph}\label{thm:3ch1dim}
{\sl Let $\mathrm{Ch}(X_1,X_2,X_3)$ be a $3$-chambar in one variable. 

In the real case $\mathrm{Ch}(X_1,X_2,X_3)$ 
is constant $(${\it i.e.} the $X_i$'s are distinct constant 
vector fields$)$.

In the complex case 
\begin{itemize}
\item[$\diamond$] either $\mathrm{Ch}(X_1,X_2,X_3)$ 
is constant, 
\item[$\diamond$] or 
$\mathrm{Ch}(X_1,X_2,X_3)=\mathrm{Ch}\Big(a(x)\frac{\partial}{\partial x},\mathbf{j}a(x)\frac{\partial}{\partial x},\mathbf{j}^2a(x)\frac{\partial}{\partial x}\Big)$, where 
$\mathbf{j}^3=1$, and $a(x)=\sqrt{\lambda x+\mu}$ with $\lambda\in\mathbb{C}^*$, 
$\mu\in\mathbb{C}$.
\end{itemize}}
\end{theoalph}

Note that the classification
implies that the global $3$-chambars in one variable have no 
singularities where they are defined; this is not 
the case in higher dimensions (consider the 
nilpotent linear cases). Whereas $2$-chambars
and $3$-chambars on an open subset of $\mathbb{C}$
are rigid the $4$-chambars are not. The classification
of $p$-chambars on $\mathbb{C}$ for $p\geq 4$ is a 
difficult problem in particular because of irreducibility
problems. Nevertheless we obtain 
interesting properties of such chambars.

\smallskip

In \S \ref{sec:homogcham} we deal with chambars generated by 
homogeneous vector fields (homogeneous chambars). 
Among other results we will see the classification of
homogeneous chambars of degree $2$ (Theorem \ref{thm:hom3cham}):

\begin{theoalph}
{\sl Let $\mathrm{Ch}(X_1,X_2,X_3)$ be an homogeneous 
$3$-chambar of $\mathbb{C}^2$ of degree $2$. Then,
after a change of variables, $X_i$ can be written 
as $a_iy^2\frac{\partial}{\partial x}$, and the
$a_i$'s satisfy: $a_1+a_2+a_3=0$. 
In particular, any homogeneous 
$3$-chambar of $\mathbb{C}^2$ of degree $2$ is 
rigid.}
\end{theoalph}

\tableofcontents

\subsection*{Acknowledgement} We warmly thank the referee 
for his careful reading.

\section{Remarks and examples}\label{sec:remeg}

Let $\mathcal{U}$ be a connected open subset of $\mathbb{R}^n$ 
(resp. $\mathbb{C}^n$). Denote by $\mathcal{O}(\mathcal{U})$
the ring of analytic (resp. holomorphic) functions
and by $\chi(\mathcal{U})$ the $\mathcal{O}(\mathcal{U})$-module 
of vector fields on $\mathcal{U}$. We denote
also by $\mathcal{O}(\mathbb{C}^n,a)$ and 
by $\chi(\mathbb{C}^n,a)$ the germs of the 
previous spaces at $a\in\mathcal{U}$.
Let $X_1$, $X_2$, $\ldots$, $X_p$, 
$Y_1$, $Y_2$, $\ldots$, $Y_q$ be some
analytic or holomorphic vector fields on
$\mathcal{U}$.
If the $p$-tuple $(X_1,\,X_2,\,\ldots,\,X_p)$ 
and the $q$-tuple $(Y_1,\,Y_2,\,\ldots,\,Y_q)$ 
satisfy the barycentric property, then the $(p+q)$-tuple 
$(X_1,\,X_2,\,\ldots,\,X_p,\,Y_1,\,Y_2,\,\ldots,\,Y_q)$
satisfy the barycentric property. This type of example is
called a \textsl{reducible chambar}. A chambar 
is \textsl{irreducible} if it is not reducible.

\subsection{Elementary examples and their variants}\label{subsec:elex}

The most elementary example is the exam\-ple of 
constant vector fields. Let $v_1$, $v_2$, $\ldots$, 
$v_p$ be $p$ distinct constant vector fields on $\mathbb{R}^n$
(resp.~$\mathbb{C}^n$) such that 
\[
v_1+v_2+\ldots+v_p=0.
\]
The translation flows $T_t^{v_k}(x)=x+tv_k$
satisfy the barycentric property
\[
\displaystyle\sum_{k=1}^pT_t^{v_k}(x)=
\displaystyle\sum_{k=1}^p(x+tv_k)=\displaystyle\sum_{k=1}^px+\displaystyle\sum_{k=1}^ptv_k=px+t\times 0=px
\]
and the vector fields $v_1$, $v_2$, $\ldots$, $v_p$
define a $p$-chambar. Such a chambar is called a \textsl{constant $p$-chambar}.
The trajectories of the $v_k$ are straight lines. 
The constant chambar $(v_1,v_2,\ldots,v_p)$
is reducible if and only if there is a subfamily
$(v_{j_1},v_{j_2},\ldots,v_{j_\ell})$ 
such that $\displaystyle\sum_{k=1}^\ell v_{j_k}=0$.

\smallskip

Let us give a simple variant of this example. Fix some coordinates
\[
(x,y)=(x_1,x_2,\ldots,x_q,y_1,y_2,\ldots,y_{n-q});
\]
take $p$ vector fields
\[
X_k=f_1^k(x)\frac{\partial}{\partial y_1}+f_2^k(x)\frac{\partial}{\partial y_2}+\ldots+f_{n-q}^k(x)\frac{\partial}{\partial y_{n-q}}
\]
where the $f_i^k$'s denote some analytic functions. Assume
that
\[
X_1+X_2+\ldots+X_p=0.
\]
The $X_k$'s satisfy the barycentric property since for any value of the 
parameter $x$ the $X_k$ are constant vector fields in the linear
subspaces $x=$ constant.

\smallskip

We can enrich this family of examples as follows. On the open 
subset $\mathcal{U}$ consider a regular foliation $\mathcal{F}$
of codimension $q$ whose leaves are of the form $A\cap\mathcal{U}$
where the $A$'s are affine subspaces of codimension $q$. Take now
analytic vector fields $X_k$ constant on any leaf of $\mathcal{F}$
and such that $X_1+X_2+\ldots+X_p=0$. Then 
$(X_1,\,X_2,\,\ldots,\,X_p)$ is a $p$-chambar. 

\smallskip

These examples play an important role in the article.

Another kind of construction that will be used is the formula expressing the flow of a vector field. Let $X=\displaystyle\sum_{k=1}^nA_k(x)\frac{\partial}{\partial x_k}$ be an analytic vector field on an open subset $\mathcal{U}$ of $\mathbb{R}^n$ or $\mathbb{C}^n$, considered as a derivation on~$\mathcal{O}(\mathcal{U})$: if $f\in\mathcal{O}(\mathcal{U})$, then 
\[
X(f)=\displaystyle\sum_{k=1}^nA_k\,\frac{\partial f}{\partial x_k}.
\]

Let $(t,x)\mapsto\varphi_t(x)$ be the flow of $X$. For $x\in\mathcal{U}$ fixed set $h(t)=f(\varphi_t(x))$. The Taylor series of $h$ at $t=0$ is of the form $h(t)=h(0)+\displaystyle\sum_{k=1}^\infty\frac{h^{(k)}(0)}{k!}t^k$.

On the other hand, $h(0)=x$ and $h^{(k)}(0)=X^k(f)$. In particular we get 
\[
f(\varphi_t(x))=x+\displaystyle\sum_{k\geq 1}\frac{1}{k!}X^k(f)(x)t^k. 
\]
If we specialize the above formula doing $f(x)=x_j$, the $j$-th coordinate of $x=(x_1,x_2,\ldots,x_n)$, then $\varphi_t(x)=\big(\varphi^1_t(x),\varphi^2_t(x),\ldots,\varphi^n_t(x)\big)$ where
\begin{equation}\label{eq:al}
\varphi^j_t(x)=x_j+\displaystyle\sum_{k\geq 1}\frac{1}{k!}X^k(x_j)t^k
\end{equation}
Formula (\ref{eq:al}) will appear in some examples. Let us now give a consequence
of (\ref{eq:al}):

\begin{pro}\label{pro:csqeqal}
{\sl Let $\mathcal{U}\subset\mathbb{C}^n$ be an open
subset. Let $X_1$, $X_2$, $\ldots$, $X_p$ be some distinct elements of $\chi(\mathcal{U})$. 
Then $X_1$, $X_2$, $\ldots$, $X_p$ define a $p$-chambar if and only if for any 
$1\leq j\leq n$
\[
\displaystyle\sum_{k=1}^pX_k^\ell(x_j)=0\qquad\forall\,\ell\geq 1
\]
where $x_j$ denotes the $j$-th coordinate of 
$x=(x_1,x_2,\ldots,x_n)$.}
\end{pro}

\subsection{Barycentric property and integrability}

Let $\mathrm{Ch}(X_1,X_2,\ldots,X_p)$ be a $p$-chambar.
Exam\-ples seen in \S\ref{subsec:elex} and $2$-chambars
may suggest that the Pfaff system generated by 
$X_1$, $X_2$, $\ldots$, $X_p$ is an integrable 
system, {\it i.e.} tangent to a foliation. 
The following example of $3$-chambar in dimension 
$3$ shows that this is not the case. Let us 
consider 
\begin{align*}
& X_1=-2\frac{\partial}{\partial x_1}+\frac{\partial}{\partial x_3}, && X_2=\frac{\partial}{\partial x_1}+x_1\frac{\partial}{\partial x_2}+\frac{\partial}{\partial x_3}, &&
X_3=\frac{\partial}{\partial x_1}-x_1\frac{\partial}{\partial x_2}-2\frac{\partial}{\partial x_3}.
\end{align*}
The flows of the $X_i$ are 
\begin{align*}
& \exp tX_1=(x_1-2t,x_2,x_3+t), \\
& \exp tX_2=\Big(x_1+t,x_2+tx_1+\frac{t^2}{2},x_3+t\Big), \\
& \exp tX_3=(x_1+t,x_2-x_1t-\frac{t^2}{2},x_3-2t). &&
\end{align*}
The barycentric property is satisfied; the leaves of $X_1$ are lines and the generic
leaves of $X_2$ and $X_3$ are parabolas. 
Let $\omega=-x_1dx_1+dx_2-2x_1dx_3$. Then 
$\omega(X_i)=0$, so $\omega$
defines the Pfaffian system
associated to the $X_i$. 
A direct computation yields to 
\[
\omega\wedge \mathrm{d}\omega=2\mathrm{d} x_1\wedge \mathrm{d} x_2\wedge\mathrm{d} x_3,
\]
{\it i.e.} the $2$-plane field associated to $\omega$
is a contact structure hence is not integrable.

\subsection{Fundamental example in dimension $1$ and generalization}\label{subsec:fundex}

Let us consider the translation flow
$\psi_t(x)=x+t$ on $\mathbb{C}$. Let
$\nu$ be an integer $\geq 2$. Denote by 
$x^{\frac{1}{\nu}}$ the principal branch
of the $\nu$-th root. Then 
\[
\varphi_{\nu,t}(x)=\big(\psi_t(x^{\frac{1}{\nu}})\big)^\nu=\big(x^{\frac{1}{\nu}}+t\big)^\nu
\]
defines a flow, at least in a neighborhood
of $1$ since it is a conjugate of the 
translation flow. This flow is polynomial in the 
time $t$ and corresponds 
to the vector field
\[
Z_\nu=\nu x^{\frac{\nu-1}{\nu}}\frac{\partial}{\partial x}=\nu\frac{x}{x^{\frac{1}{\nu}}}\frac{\partial}{\partial x}
\]
well defined at least in a neighborhood 
of $1$. 

Let $\sigma$ be a primitive $(\nu+1)$-th root
of unity. Then 
\[
\varphi_{\nu,\sigma t}(x)=\big(x^{\frac{1}{\nu}}+\sigma t\big)^\nu
\]
is the flow of the vector field
\[
\sigma Z_\nu=\nu\sigma\frac{x}{x^{\frac{1}{\nu}}}\frac{\partial}{\partial x}.
\]

Of course 
$\displaystyle\sum_{p=0}^\nu\sigma^p\cdot Z_\nu=0$ 
and 
\[
\displaystyle\sum_{p=0}^\nu \Big(x^{\frac{1}{\nu}}+\sigma^pt\Big)^\nu=
\sum_{p=0}^\nu\sum_{k=0}^\nu\binom{\nu}{k}x^{\frac{\nu-k}{\nu}}\sigma^{pk}t^k=\sum_{k=0}^\nu\Big(\sum_{p=0}^\nu\sigma^{pk}\Big)t^k\binom{\nu}{k}x^{\frac{\nu-k}{\nu}}=(\nu+1)x.
\]

\medskip

We can thus state 

\begin{pro}\label{pro:zn}
{\sl Let $Z_\nu$ be the vector field defined in a neighborhood of 
$1$ by 
\[
Z_\nu=\nu x^{\frac{\nu-1}{\nu}}\frac{\partial}{\partial x}=\nu\frac{x}{x^{\frac{1}{\nu}}}\frac{\partial}{\partial x}.
\]

The $(\nu+1)$-tuple $(Z_\nu,\sigma Z_\nu,\ldots,\sigma^\nu Z_\nu)$ is an irreducible $(\nu+1)$-chambar in a neighborhood of $1$.}
\end{pro}

One can conjugate a chambar by an affine map; 
hence
\[
\left((\lambda x+\mu)^{\frac{\nu-1}{\nu}}\frac{\partial}{\partial x},\sigma(\lambda x+\mu)^{\frac{\nu-1}{\nu}}\frac{\partial}{\partial x},\sigma^2(\lambda x+\mu)^{\frac{\nu-1}{\nu}}\frac{\partial}{\partial x},\ldots,\sigma^\nu(\lambda x+\mu)^{\frac{\nu-1}{\nu}}\frac{\partial}{\partial x}\right)
\]
produces a $(\nu+1)$-chambar where it makes 
sense. 

For $\nu=2$ the previous construction gives 
the flow $\varphi_{2,t}(x)=x+2t\sqrt{x}+t^2$ associated
to the vector field 
$Z_2=2\sqrt{x}\frac{\partial}{\partial x}$ 
and the $3$-chambar 
$\mathrm{Ch}(Z_2,\mathbf{j}Z_2,\mathbf{j}^2Z_2)$, 
$\mathbf{j}^3=1$, 
but also its affine conjugates.

An immediate generalization in any dimension is the following.
Consider $P(x)=(P_1(x),P_2(x),\ldots,P_n(x))$ such that
\begin{itemize}
\item[$\diamond$] $P_j\in\mathbb{C}[x_1,x_2,\ldots,x_n]$, $\deg P_1=\nu\geq 2$ and $\deg P_j\leq \nu$,

\item[$\diamond$] $P(0)=0$,

\item[$\diamond$] and $DP(0)=\rho\cdot\mathrm{id}$ where $\mathrm{id}$ is the 
identity of $\mathbb{C}^n$ and $\vert\rho\vert>1$.
\end{itemize}

There exists a neighborhood $U$ of $0\in\mathbb{C}^n$ such that 
$V=P(U)\supset U$ and $P_{\vert U}$ has an inverse $\phi\colon V\to U$.
To any $a=(a_1,a_2,\ldots,a_n)\in\mathbb{C}^n$ we can associate
a flow defined in a neighborhood of $(0,0)\in\mathbb{C}\times\mathbb{C}^n$ by
\[
\varphi_t(x)=P\big(\phi(x)+ta\big).
\]
The vector field associated to this flow is
\begin{equation}\label{eq:flowvf}
X(x)=DP(\phi(x))\cdot a
\end{equation}

\begin{pro}\label{pro:irrcham}
{\sl Let $X$ be as in $(\ref{eq:flowvf})$ and let $\sigma$ be a 
primitive $(\nu+1)$-th root of unity. Then the $(\nu+1)$-tuple 
$(X,\sigma X,\ldots,\sigma^\nu X)$ is an irreducible 
$(\nu+1)$-chambar in a neighborhood of $0\in\mathbb{C}^n$.}
\end{pro}

\begin{proof}[{\sl Proof}]
Since $P$ has degree $\nu$ 
\begin{eqnarray*}
\varphi_t(x)&=&P\big(\phi(x)+ta\big)\\
&=&P(\phi(x))+tDP(\phi(x))\cdot a+\displaystyle\sum_{j=2}^\nu \frac{t^j}{j!}D^{(j)}P(\phi(x))\cdot a\\
&=&x+tH_1(x,a)+\displaystyle\sum_{j=2}^\nu t^jH_j(x,a)
\end{eqnarray*}
where $H_j(x,a)$ is homogeneous of degree $j$ with respect to 
$a\in\mathbb{C}^n$. Hence the flow of $\sigma^kX$~is
\[
\varphi_{\sigma^k\cdot t}(x)=x+\sigma^ktH_1(x,a)+\displaystyle\sum_{j=2}^\nu\sigma^{jk}t^jH_j(x,a)
\]
and so 
\[
\displaystyle\sum_{k=0}^\nu\varphi_{\sigma^k\cdot t}(x)=\displaystyle\sum_{k=0}^\nu\Big(x+\sigma^ktH_1(x,a)+\displaystyle\sum_{j=2}^\nu\sigma^{jk}t^jH_j(x,a)\Big)=(\nu+1)x
\]
because $\displaystyle\sum_{k=0}^\nu\sigma^{jk}=0$ if $1\leq j\leq \nu$.
\end{proof}

\begin{rem}
The construction produces vector fields $X$
whose flow $\exp tX$ is
polynomial in the variable time $t$.
\end{rem}

\begin{eg}
A global example of this kind (Proposition \ref{pro:irrcham})
can be given by a polynomial 
diffeomorphism $P\colon\mathbb{C}^n\to\mathbb{C}^n$. 
For instance
\[
P(x_1,x_2,\ldots,x_n)=(x_1,x_2+q_2(x_1),x_3+q_3(x_1,x_2),\ldots,x_n+q_n(x_1,x_2,\ldots,x_{n-1}))
\]
where $q_j\in\mathbb{C}[x_1,x_2,\ldots,x_{j-1}]$, $2\leq j\leq n$.

\end{eg}

\medskip

As a particular example, consider the polynomial
diffeomorphism of $\mathbb{C}^2$ 
\[
\phi(x,y)=(x+y^2,y).
\]
Conjugating the flow
\[
(x+a_kt,y+b_kt)\qquad a_k,\,b_k\in\mathbb{C}
\]
with $\varphi$ we get the flow
\[
\phi_k^t=\big(x+a_kt+2b_kty+b_k^2t^2,y+b_kt\big);
\]
one can check that it is the flow of the affine 
vector field
\[
X_k=\Big(a_k+2b_ky\Big)\frac{\partial}{\partial x}+b_k\frac{\partial}{\partial y}.
\]
Remark that this flow is polynomial in the time $t$.

As soon as $b_k\not=0$ the trajectories are 
the parabola
\[
f_k=a_ky+b_ky^2-b_kx=\text{constant}.
\]
For $p\geq 3$ if we choose $a_1$, $a_2$, $\ldots$, 
$a_p$, $b_1$, $b_2$, $\ldots$, $b_p$ such that 
\begin{equation}\label{eq:hidi}
a_1+a_2+\ldots+a_p=b_1+b_2+\ldots+b_p=b_1^2+b_2^2+\ldots+b_p^2=0
\end{equation}
then the $X_k$ satisfy the barycentric property and 
produce a $p$-chambar. For a generic choice of the 
parameters $a_k$ and $b_k$ the $X_k$ are not 
$\mathbb{C}$-colinear. Note that for $p=3$ if 
$(\ref{eq:hidi})$ holds, then the web
$\mathrm{W}(X_1,X_2,X_3)$ is an hexagonal web 
(\emph{see for instance} \cite{PereiraPirio}) since
$f_1+f_2+f_3=0$.

\subsection{Polynomial vector fields that satisfy the barycentric property}

\begin{pro}\label{pro:dim1pol}
{\sl In dimension $1$ the polynomial vector
fields that satisfy the barycentric 
property are the constant vector fields
\[
a_k\frac{\partial}{\partial x}
\]
with $a_k\in\mathbb{C}^*$ and 
$\displaystyle\sum_{k=1}^pa_k=0$.}
\end{pro}

\begin{proof}[{\sl Proof}]
The proof is based on Proposition \ref{pro:csqeqal}.
Let $X=P(x)\frac{\partial}{\partial x}$ where $P\in\mathcal{O}(\mathbb{C})$
is viewed as a derivation on $\mathcal{O}(\mathbb{C})$.
According to $(\ref{eq:al})$ the flow $\varphi_t$ of $X$ 
is 
\[
\varphi_t(x)=x+\displaystyle\sum_{k\geq 1}\frac{1}{k!}X^k(x)t^k.
\]
If $P\in\mathbb{C}[x]$ is a polynomial of degree $d\geq 1$, 
then $X^k(x)$ is also a polynomial for any $k\geq 1$. Let us 
write $X^k(x)$ as $X^k(x)=\displaystyle\sum_{j=0}^{d(k)}a_j^kx^j$.
If we set $d(\ell):=\deg(X^\ell(x))$, then 
\begin{itemize}
\item[(1)] since $\deg(X)=d$, then $a_d^1\not=0$;

\item[(2)] $d(\ell)=(d-1)\ell+1$ because $d(\ell+1)=\deg(X(x))+d(\ell)-1=d+d(\ell)-1$;

\item[(3)] the equality $a_{d(\ell+1)}^{\ell+1}=d(\ell)a_d^1a_{d(\ell)}^\ell$ holds.
\end{itemize}

By recurrence we get from $(3)$ that $a_{d(\ell)}^\ell=A(\ell)(a_d^1)^\ell$ where
\begin{itemize}
\item[(4)] $A(1)=1$ and $A(\ell+1)=d(\ell)A(\ell)$ for $\ell\geq 1$.
\end{itemize}

On the one hand if $d=0$, then $X(x)\not=0$ and $X^\ell(x)=0$ for all
$\ell\geq 2$. On the other hand it follows from $(2)$, $(3)$ and $(4)$ 
that if $d\geq 1$, then $d(\ell)\geq 1$ and $A(\ell)\geq 1$ for all 
$\ell\geq 1$. 

\smallskip

Now assume that $(X_1,X_2,\ldots,X_p)$ is a polynomial $p$-chambar on 
$\mathbb{C}$. Let $d=\displaystyle\max_{1\leq j\leq p}\deg(X_j)$. 
Suppose by contradiction that $d\geq 1$. Without lost of generality
we can assume that 
\[
\{j\,\vert\,\deg(X_j)=d\}=\{1,\,2,\,\ldots,\,q\}\subset\{1,\,2,\ldots,\,p\}.
\]
Set $X_k=P_k(x)\frac{\partial}{\partial x}$ where 
$P_k(x)=\displaystyle\sum_{j=0}^da_{kj}x^j$, $1\leq k\leq p$,
where 
\[
a_{jd}\not=0\text{ if } 1\leq j\leq q \quad \text{and} \quad a_{jd}=0\text{ if }q<j\leq p. 
\]
\begin{claim}
For any $\ell\geq 1$ we have 
\[
(a_{1d})^\ell+(a_{2d})^\ell+\ldots+(a_{qd})^\ell=0.
\]
\end{claim}
The statement follows from the Claim 
(indeed if $(a_{1d})^\ell+(a_{2d})^\ell+\ldots+(a_{qd})^\ell=0$ for any $\ell\geq 1$, then $a_{1d}=a_{2d}=\ldots=a_{qd}=0$). Let us 
now justify it:

\begin{proof}[{\sl Proof of the Claim}]
Set $d(k,\ell)=\deg(X_k^\ell(x))$, $1\leq k\leq p$. Note that:
\begin{itemize}
\item[$\diamond$] if $d=1$, then $d(k,\ell)=1$ for all $1\leq k\leq q$ and all 
$\ell\geq 1$; furthermore if $q<k\leq p$, then $X_k^\ell(x)=0$ for all $\ell\geq 2$.

\item[$\diamond$] id $d>1$ and $1\leq k\leq q$, then $d\leq d(k,\ell)=(d-1)k+1$ and so $d(k,\ell)<d(k,\ell+1)$ for all $\ell\geq 1$. Moreover if $q<k\leq p$, then either $d(k,\ell)<(d-1)k+1$ or $X_k^\ell(x)=0$ for all $\ell\geq 2$.
\end{itemize}
Given $1\leq k\leq p$ let $a(k,\ell)$ be the  coefficient of $x^{d(k,\ell)}$ 
in the polynomial $X_k^\ell(x)$. If follows from the above computations
that
\begin{itemize}
\item[$\diamond$] if $1\leq k\leq q$, then $a(k,\ell)=A(\ell)(a_{kd})^\ell$ where 
$A(\ell)\not=0$,

\item[$\diamond$] if $q<k\leq p$, then $a(k,\ell)=0$.
\end{itemize}
According to Proposition \ref{pro:csqeqal} we get that
\[
X_1^\ell(x)+X_2^\ell(x)+\ldots+
X_p^\ell(x)=0
\]
implies 
\[
A(\ell)\Big((a_{1d})^\ell+(a_{2d})^\ell+\ldots+(a_{qd})^\ell=0\Big)=0;
\]
as $A(\ell)\not=0$ we finally obtain that $(a_{1d})^\ell+(a_{2d})^\ell+\ldots+(a_{qd})^\ell=0$.
\end{proof}
\end{proof}

\begin{rem}
If $p=3$, then Proposition \ref{pro:dim1pol}
is a consequence of Theorem 
\ref{thm:loc3bfc}.
\end{rem}

\begin{rem}
If $X$ is a holomorphic vector field on the 
Riemann sphere 
$\overline{\mathbb{C}}=\mathbb{C}\cup\{\infty\}$,
then in the affine chart $\mathbb{C}$ there exists a polynomial function $a$ 
of degree $\leq 2$ such that 
$X=a(x)\frac{\partial}{\partial x}$.
The only $p$-tuple of global vector fields
that satisfy the barycentric property
in this chart are
 the constant vector fields.
\end{rem}

\subsection{Examples produced by those of dimension $1$}\label{sec:exconj}

We need a definition:

\begin{defi}
A $p$-chambar of the form
$\mathrm{Ch}(a_1X,a_2X,\ldots,a_pX)$, 
with $a_i$ constant, is called  
\textsl{rigid}.
\end{defi}

\smallskip

Propositions  \ref{pro:zn} and \ref{pro:irrcham}  give examples of 
rigid $p$-chambars.

\smallskip

Let us give a construction presented in dimension $2$
for simplicity but that can be gene\-ralized in any 
dimension $n$ and for any $p$. 

Consider the vector field 
$X(x)=2\sqrt{x}\frac{\partial}{\partial x}$ that 
induces the flow 
$\varphi_t(x)=x+2t\sqrt{x}+t^2$, a 
special case of \S \ref{subsec:fundex}. A first 
$3$-chambar in dimension $2$ is
\[
\mathrm{Ch}\big(X(x)+X(y),\mathbf{j}(X(x)+X(y)),\mathbf{j}^2(X(x)+X(y))\big)
\]
which is rigid. Similarly one can consider
\[
\mathrm{Ch}\big(X(x)+X(y),\mathbf{j}X(x)+\mathbf{j}^2X(y),\mathbf{j}^2X(x)+\mathbf{j}X(y)\big)
\]
which is non-rigid. These examples are well 
defined on any simply connected open subset 
that do not intersect the axis $x=0$ and $y=0$.

\smallskip
Let us now give an example of a non-rigid irreducible 
$4$-chambar still in dimension $2$
\[
\mathrm{Ch}\Big(X(x),\mathbf{j}X(x)+X(y),\mathbf{j}^2X(x)+\mathbf{j}X(y),\mathbf{j}^2X(y)\Big)
\]
that can be generalized to a $5$-chambar as
follows
\[
\mathrm{Ch}\Big(X(x),\mathbf{j}X(x),\mathbf{j}^2X(x)+X(y),\mathbf{j}X(y),\mathbf{j}^2X(y)\Big).
\]

\begin{eg}
Another way to obtain examples is 
by taking the real part of a 
complex $p$-chambar on $\mathbb{C}^n$.
For instance, if we set 
$z=x+\mathbf{i}y$, then 
$\frac{d}{dz}=\frac{1}{2}\left(\frac{d}{dx}-\mathbf{i}\frac{d}{dy}\right)$,
\[
\sqrt{2}\sqrt{z}=\underbrace{\sqrt{\sqrt{x^2+y^2}+x}}_{A(x,y)}+\mathbf{i}\underbrace{\sqrt{\sqrt{x^2+y^2}-x}}_{B(x,y)}
\]
and
\[
\mathrm{Re}\left(\sqrt{z}\,\frac{d}{dz}\right)=\frac{1}{2\sqrt{2}}\left(A(x,y)\,\frac{d}{dx}+B(x,y)\,\frac{d}{dy}\right).
\]
The three vector fields 
$\mathrm{Re}\left(\sqrt{z}\,\frac{d}{dz}\right)$, 
$\mathrm{Re}\left(\mathbf{j}\,\sqrt{z}\,\frac{d}{dz}\right)$, 
$\mathrm{Re}\left(\mathbf{j}^2\,\sqrt{z}\,\frac{d}{dz}\right)$ 
give a real $3$-chambar but if we consider $x$, $y$ 
as complex variables we get a $3$-chambar on a suitable 
open set of $\mathbb{C}^2$.

Let us remark that we can iterate this process: take a 
chambar on $\mathbb{C}^n$, its real part gives a 
chambar on $\mathbb{R}^{2n}$ whose complexification 
is a chambar on $\mathbb{C}^{2n}$ and so on...
\end{eg}

\smallskip

\subsection{Examples associated to some polynomial flows in $t$}\label{subsec:exples}

\subsubsection{Polynomial examples}

Let $P=p_0+p_1x+\ldots+p_Nx^\nu$ be a polynomial 
of degree $\nu$. Consider the vector field
\[
X=a\frac{\partial}{\partial x}+P(x)\frac{\partial}{\partial y}
\]
where $a\in\mathbb{C}^*$. Its flow is polynomial in $t$:
\[
\varphi_t(x,y)=\left(x+at,y+\displaystyle\sum_{k=0}^\nu p_k\left(\frac{(x+at)^{k+1}}{a(k+1)}-\frac{x^{k+1}}{a(k+1)}\right)\right)
\]
that can we rewritten
\[
\varphi_t(x,y)=\Big(x+at,y+\widetilde{P}_a(x+at)-\widetilde{P}_a(x)\Big)
\]
where $\widetilde{P}_a(y)=\displaystyle\sum_{k=0}^\nu p_k\frac{y^{k+1}}{a(k+1)}$.

Let us consider $p$ vector fields $X_1$, $X_2$, $\ldots$, 
$X_p$ of the following form
\[
X_k=a_k\frac{\partial}{\partial x}+P_k(x)\frac{\partial}{\partial y}.
\]
The barycentric property is equivalent to 
\begin{equation}\label{eq:polbp1}
\displaystyle\sum_{k=1}^pa_k=0
\end{equation}
and 
\begin{equation}\label{eq:polbp2}
\displaystyle\sum_{k=1}^p \widetilde{P}_{k,a_k}(x+a_kt)-\widetilde{P}_{k,a_k}(x)=0
\end{equation}
Note that $(\ref{eq:polbp2})$ holds if and only if 
\[
\frac{\partial}{\partial t}\left(\displaystyle\sum_{k=1}^p \widetilde{P}_{k,a_k}(x+a_kt)\right)=0
\]
if and only if
\begin{equation}\label{eq:polbp3}
\displaystyle\sum_{k=1}^\nu P_k(x+a_kt)=0
\end{equation}
As soon as we have fixed the constants $a_1$, $a_2$, 
$\ldots$, $a_p$ the equality $(\ref{eq:polbp3})$
is a linear system in the coefficients of the 
polynomials $P_k$, system that sometimes has non-trivial
solutions. 

Consider for instance the case $p=3$ and $\nu=2$. Set
\begin{align*}
& P_1=\alpha_0+\alpha_1x+\alpha_2x^2, && P_2=\beta_0+\beta_1x+\beta_2x^2, &&
P_3=\gamma_0+\gamma_1x+\gamma_2x^2.
\end{align*}

Conditions $(\ref{eq:polbp1})$ and $(\ref{eq:polbp3})$ are 
equivalent to 
\begin{align*}
& (I)\,\,\left\{
\begin{array}{ll}
a_1+a_2+a_3=0\\
\alpha_0+\beta_0+\gamma_0=0\\
\alpha_1+\beta_1+\gamma_1=0\\
\alpha_1a_1+\beta_1a_2+\gamma_1a_3=0
\end{array}
\right.
&& (II)\,\,\left\{
\begin{array}{ll}
\alpha_2+\beta_2+\gamma_2=0\\
\alpha_2a_1+\beta_2a_2+\gamma_2a_3=0\\
\alpha_2a_1^2+\beta_2a_2^2+\gamma_2a_3^2=0
\end{array}
\right.
\end{align*}
In other words $(\ref{eq:polbp1})$ and $(\ref{eq:polbp3})$
give seven equations in the parameters space 
$\alpha$, $\beta$, $\gamma$, $a$ of dimension~$12$.
The set of solutions is not irreducible. 
Assume that the parameters $a=\underline{a}$ satisfies
$\underline{a_1}\not=\underline{a_2}\not=\underline{a_3}$. 
Then in a neighborhood of $a=\underline{a}$ the system
$(II)$ is a Vandermonde one so has for solution 
$\alpha_2=\beta_2=\gamma_2=0$. Then 
$(I)$ and $(II)$ are equivalent to 
\[
\left\{
\begin{array}{lllll}
a_1+a_2+a_3=0\\
\alpha_0+\beta_0+\gamma_0=0\\
\alpha_1+\beta_1+\gamma_1=0\\
\alpha_1a_1+\beta_1a_2+\gamma_1a_3=0\\
\alpha_2=\beta_2=\gamma_2=0
\end{array}
\right.
\]
that defines a quadric of dimension $12-7=5$. 
But there are solutions such that two of the 
$\underline{a_i}$ are equal. For instance if 
$\underline{a_1}=\underline{a_2}=\underline{a_3}=0$, 
then $(I)$ and $(II)$ are equivalent to 
\[
\underline{a_1}=\underline{a_2}=\underline{a_3}=\alpha_0+\beta_0+\gamma_0=\alpha_1+\beta_1+\gamma_1=\alpha_2+\beta_2+\gamma_2=0
\]
which is a linear space of dimension $12-6=6$.

Hence the set $\Sigma$ of vector fields of this
type satisfying the barycentric pro\-perty
is not irreducible. 

\subsubsection{Birational examples}

Take $(a_1,a_2,\ldots,a_p)$ a $p$-tuple of $\mathbb{C}^n$ and set
for $1\leq k\leq p$
\[
a_k=(a_{k,1},a_{k,2},\ldots,a_{k,n}).
\]
Consider the translation flow 
\[
T_t^{a_k}(x_1,x_2,\ldots,x_n)=(x_1+a_{k,1}t,x_2+a_{k,2}t,\ldots,x_n+a_{k,n}t).
\]
Denote by $\psi$ the blow-up
\[
\psi\colon(x_1,x_2,\ldots,x_n)\dashrightarrow(x_1,x_1x_2,\ldots,x_1x_n).
\]
The lift $F_t^k$ of $T_t^{a_k}$ by $\psi$ can be 
written
\begin{eqnarray*}
F_t^k(x)&=&\psi\circ T_t^{a_k}\circ\psi^{-1}(x)\\
&=& \left(x_1+a_{k,1}t,\big(x_1+a_{k,1}t\big)\left(\frac{x_2}{x_1}+a_{k,2}t\right),\ldots,\big(x_1+a_{k,1}t\big)\left(\frac{x_n}{x_1}+a_{k,n}t\right)\right).
\end{eqnarray*}
The condition $\displaystyle\sum_{k=1}^pF_t^k(x)=px$ is satisfied if 
\begin{itemize}
\item[$\diamond$] for any $1\leq \ell\leq n$
\[
\displaystyle\sum_{k=1}^pa_{k,\ell}=0
\]
\item[$\diamond$] and for any $2\leq \ell\leq n$
\[
\displaystyle\sum_{k=1}^pa_{k,1}a_{k,\ell}=0.
\]
\end{itemize}

\begin{rem}
In the previous examples we assume that the 
$a_k$'s are not all zero. Up to a li\-near 
conjugation (such a conjugation preserves a 
barycentric property) we can assume that 
$a_1=(1,0,0,\ldots,0)$. The previous 
conditions can be rewritten
\[
\left\{
\begin{array}{ll}
\displaystyle\sum_{k=1}^p a_{k,\ell}=0\qquad 1\leq\ell\leq n\\
a_{1,\ell}=0\qquad 2\leq\ell\leq n
\end{array}
\right.
\]
that thus form a linear subspace of the space
of coefficients $a_{j,i}$. These examples of 
$p$-chambars are given by birational flows 
quadratic in the time $t$ (\emph{see}
\cite{CerveauDeserti} for other examples).
\end{rem}

\subsection{Examples of chambars whose flows are 
non-algebraic/non-polynomial in $t$}

Let $k$ be an integer; consider $q_k$ vector fields 
of the following form
\[
X_k^j=a_k\frac{\partial}{\partial x}+b_{k,j}\mathrm{e}^{\lambda_kx}\frac{\partial}{\partial y}\qquad 1\leq j\leq q_k
\]
where $a_k$, $b_{k,j}$ and $\lambda_k$ belong to 
$\mathbb{C}^*$. The flows of $X_k^j$ is
\[
(\exp tX_k^j)(x,y)=\Big(x+a_kt,y+\frac{b_{k,j}}{\lambda_ka_k}\mathrm{e}^{\lambda_kx}(\mathrm{e}^{\lambda_ka_kt}-1)\Big)
\]
Set $\ell=\displaystyle\sum_{k=1}^pq_k$. The $\ell$ 
vector fields $X_k^j$ form a $\ell$-chambar if and
only if for any $1\leq k\leq p$ the following 
equalities hold
\begin{align*}
&\displaystyle\sum_{k=1}^pq_ka_k=0, && \displaystyle\sum_{j=1}^{q_k}b_{k,j}=0.
\end{align*}
Contrary to the previous example the flows 
$\exp tX_k^j$ are non-polynomial: their 
orbits are the levels of the functions
\[
\lambda_ka_ky-b_{k,j}\mathrm{e}^{\lambda_kx}.
\]
This construction starts with $\ell=4$ and 
produces global chambars on $\mathbb{C}^2$. 
It can be gene\-ralized to higher dimensions.

\subsection{Compatible diffeomorphisms}

The concept of $p$-chambar is an affine one, that is the barycentric
property is invariant under the action of the group
of affine transformations; if $\mathcal{C}$ is a local
$p$-chambar and $\phi$ a diffeomorphism, then, in 
general, $\phi_*\mathcal{C}$ is not a chambar.

\begin{prob}
Let $\mathrm{Ch}_c$ be a constant chambar;
what are the diffeomorphisms $\phi$ such that 
$\phi_*\mathrm{Ch}_c$ is a $p$-chambar ?
What is the structure of such a set
of diffeomorphisms ?
\end{prob} 

Let us give an answer to this problem in the special case $p=3$, 
$n=2$. Let $\mathrm{Ch}(X_1,X_2,X_3)$ a constant $3$-chambar in 
$\mathbb{C}^2$. We say that $\mathrm{Ch}(X_1,X_2,X_3)$ is 
\textsl{generic} if the $X_i$'s
are linearly independent. We immediately notice that a generic 
constant $3$-chambar is linearly conjugate to the "standard" 
$3$-chambar
\[
\mathrm{Ch}_0=\mathrm{Ch}\left(\frac{\partial}{\partial x},\frac{\partial}{\partial y},-\left(\frac{\partial}{\partial x}+\frac{\partial}{\partial y}\right)\right).
\]
Let $\phi$ be a local diffeomorphism; we say that $\phi$ is 
\textsl{compatible with} $\mathrm{Ch}_0$ if 
$\phi_*\mathrm{Ch}_0$ is a $3$-chambar. We have the following 
statement (recall that $\mathbf{j}$, $\mathbf{j}^2$ are the roots
of $t^2+t+1$):

\begin{thm}\label{thm:cstpchamb}
{\sl A local diffeomorphism of $\mathbb{C}^2$ is compatible 
with $\mathrm{Ch}_0$ if and only if it can be written
$L+F$ where $L$ denotes an affine inversible transformation 
and $F=(f,g)$ with
\[
f,\,g\in\langle(y+\mathbf{j}x)^2,(y+\mathbf{j}^2x)^2,xy(y-x)\rangle_\mathbb{C}.
\]}
\end{thm}

\begin{rem}
A local compatible diffeomorphism is in fact a global application, 
but not in general a global diffeomorphism.
\end{rem}

Let us first state and prove the following 
result we use in the proof of 
Theorem \ref{thm:cstpchamb}:

\begin{lem}\label{lem:tec}
{\sl If $h$ is a holomorphic function satisfying the 
P.D.E's
\begin{align*}
& \frac{\partial^2h}{\partial x^2}+\frac{\partial^2h}{\partial x \partial y}+\frac{\partial^2h}{\partial y^2}=0 && \frac{\partial^3h}{\partial x^2 \partial y}+\frac{\partial^3h}{\partial x\partial y^2}=0
\end{align*}
then $h$ is a polynomial of degree $3$ of the form
\[
h(x,y)=\alpha_0+\alpha_1x+\alpha_2y+\alpha_3(x+\mathbf{j}y)^2+\alpha_4(x+\mathbf{j}^2y)^2+\alpha_5xy(y-x)
\]
with $\alpha_0$, $\alpha_1$, $\ldots$, $\alpha_5\in\mathbb{C}$.}
\end{lem}

\begin{proof}[{\sl Proof}]
To simplify the notations let us consider the 
differential operators
\begin{align*}
& S=\frac{\partial^2}{\partial x^2}+\frac{\partial^2}{\partial x \partial y}+\frac{\partial^2}{\partial y^2} && T=\frac{\partial^3}{\partial x^2 \partial y}+\frac{\partial^3}{\partial x\partial y^2}
\end{align*}

The inclusion $\langle 1,\,x,\,y,\,(y+\mathbf{j}x)^2,\,(y+\mathbf{j}^2x)^2,\,xy(y-x)\rangle_\mathbb{C}\subset\ker(S)\cap \ker(T)$ 
is straightforward.

\medskip

Note that 
\[
\frac{\partial}{\partial x}\cdot S=\frac{\partial^3}{\partial x^3}+\frac{\partial^2}{\partial x^2}\frac{\partial}{\partial y}+\frac{\partial}{\partial x}\frac{\partial^2}{\partial y^2}=\frac{\partial^3}{\partial x^3}+T
\]
so $\ker(S)\cap\ker(T)\subset\ker\left(\frac{\partial^3}{\partial x^3}\right)$.

Similarly $\frac{\partial}{\partial y}\cdot S=\frac{\partial^3}{\partial y^3}+T$ and thus 
$\ker(S)\cap\ker(T)\subset\ker\left(\frac{\partial^3}{\partial y^3}\right)$.

As a result $\ker(S)\cap \ker(T)\subset \ker\left(\frac{\partial^3}{\partial x^3}\right)\cap \ker\left(\frac{\partial^3}{\partial y^3}\right)$. In particular if $h$ belongs to 
$\ker(S)\cap\ker(T)$, then $\frac{\partial^3h}{\partial x^3}=\frac{\partial^3h}{\partial y^3}=0$.

Let $h=\displaystyle\sum_{k,\ell}h_{k,\ell}x^ky^\ell$ be 
the Taylor series of $h$ at $(0,0)$. If $\frac{\partial^3h}{\partial x^3}=\frac{\partial^3h}{\partial y^3}=0$, then $h_{k,\ell}\not=0$
if and only if $k$, $\ell\leq 2$. However if $k=\ell=2$, 
then we have $S(x^2y^2)=2y^2+2x^2+4xy\not=0$ and so 
\[
\ker(S)\cap\ker(T)=\langle 1,\,x,\,y,\,(y+\mathbf{j}x)^2,\,(y+\mathbf{j}^2x)^2,\,xy(y-x)\rangle_\mathbb{C}.
\]
\end{proof}

\begin{proof}[{\sl Proof of Theorem \ref{thm:cstpchamb}}]
If $\phi$ is a local diffeomorphism of $\mathbb{C}^2$
compatible with $\mathrm{Ch}_0$, then the 
barycentric condition asserts that
\begin{equation}\label{eq:barycond}
\phi(x+t,y)+\phi(x,y+t)+\phi(x-t,y-t)=3\phi(x,y).
\end{equation}
We can assume that $\phi$ is defined in a neighborhood
of $(0,0)$. Let us write $\phi$ as $L+(f,g)$
where $L$ is affine and $f$, $g\in\mathcal{O}(\mathbb{C}^2,0)$ satisfy $(f,g)(0,0)=D(f,g)(0,0)=(0,0)$. 
By derivating (\ref{eq:barycond}) twice with respect
to $t$, we get that both components $f$
and~$g$ satisfy the P.D.E.
\[
\frac{\partial^2h}{\partial x^2}+\frac{\partial^2h}{\partial x\partial y}+\frac{\partial^2h}{\partial y^2}=0.
\]
The solutions of such P.D.E. are of the following 
type
\begin{equation}\label{eq:solcla}
h=\varphi_+(y+\mathbf{j}x)+\varphi_-(y+\mathbf{j}^2x)
\end{equation}
with $\mathbf{j}$, $\mathbf{j}^2$ the roots of $t^2+t+1$
and $\varphi_+$, $\varphi_-$ holomorphic in one 
variable defined on suitable domains. 

A third derivation with respect to $t$ shows
that $f$ and $g$ also satisfy the P.D.E.
\begin{equation}\label{eq:barycond2}
0=\frac{\partial^3h}{\partial x^2\partial y}+\frac{\partial^3h}{\partial x\partial y^2}=\frac{\partial^2}{\partial x\partial y}\left(\frac{\partial h}{\partial x}+\frac{\partial h}{\partial y}\right).
\end{equation}

Lemma \ref{lem:tec} allows to conclude (note that, 
with the notations of Lemma \ref{lem:tec} an 
element of $\ker S\cap\ker T$ satisfies 
relation (\ref{eq:barycond})).
\end{proof}

More generally, one can state:

\begin{thm}\label{thm:baz}
{\sl Let $f\colon \mathcal{U}\to f(\mathcal{U})\subset\mathbb{C}^n$
be a biholomorphism from the open set $\mathcal{U}\subset\mathbb{C}^n$
to $f(\mathcal{U})$, $n\geq 2$. Assume that the vector fields 
\begin{align*}
& f_*\frac{\partial}{\partial x_1},&& f_*\frac{\partial}{\partial x_2}, &&\ldots, &&f_*\frac{\partial}{\partial x_n}, && f_*\left(-\frac{\partial}{\partial x_1}-\frac{\partial}{\partial x_2}-\ldots-\frac{\partial}{\partial x_n}\right)
\end{align*}
satisfy the barycentric property. Then all the components $f_j$ of $f$ are 
polynomial.}
\end{thm}

\begin{lem}\label{claim1}
{\sl Let $h\in\mathcal{O}(\mathcal{U})$ be a holomorphic function
with the property that 
\begin{small}
\begin{equation}\label{eq:claim}
\displaystyle\sum_{j=1}^nh(x_1,x_2,\ldots,x_{j-1}x_j+t,x_{j+1},x_{j+2},\ldots,x_n)+h(x_1-t,x_2-t,\ldots,x_n-t)=(n+1)h(x_1,x_2,\ldots,x_n)
\end{equation}
\end{small}
for all $x\in\mathcal{U}$ and $t\in\mathbb{C}$ with $\vert t\vert$ 
small enough. Then $h$ satisfies the system of P.D.Es 
\[
\left\{
\begin{array}{lll}
T_2(h)=0\\
T_3(h)=0\\
\ldots
\end{array}
\right.
\]
where $T_k$ is the differential operator
\[
T_k=\frac{\partial^k}{\partial x_1^k}+\frac{\partial^k}{\partial x_2^k}+\ldots+\frac{\partial^k}{\partial x_n^k}+(-1)^k\Big(\frac{\partial}{\partial x_1}+\frac{\partial}{\partial x_2}+\ldots+\frac{\partial}{\partial x_n}\Big)^k.
\]}
\end{lem}

\begin{proof}[{\sl Proof}]
Let $e_1=(1,0,0,\ldots,0)$, $e_2=(0,1,0,0,\ldots,0)$, $\ldots$, 
$e_n=(0,0,\ldots,0,1)$ and $v=-\displaystyle\sum_{j=1}^ne_j$. 
The idea is to prove by induction on $k\geq 1$ that for 
any $t\in(\mathbb{C},0)$
\begin{equation}\label{eq:int}
\displaystyle\sum_{j=1}^n\frac{\partial^k}{\partial x_j}h(x+te_j)+(-1)^k\left(\frac{\partial}{\partial x_1}+\frac{\partial}{\partial x_2}+\ldots+\frac{\partial}{\partial x_n}\right)^kh(x+tv)=0;
\end{equation}
indeed if $t=0$ in (\ref{eq:int}), then we get 
(\ref{eq:claim}).

\medskip

Let $\varphi(t,x)=\displaystyle\sum_{j=1}^n h(x+te_j)+h(x+tv)$. 
According to $(\ref{eq:claim})$ the function $\varphi(t,x)$ 
depends only of $x$. In particular differentiating $k$ times
with respect to $t$ we get 
\[
\frac{\partial^k\varphi(t,x)}{\partial t^k}=\displaystyle\sum_{j=1}^n\frac{\partial^k}{\partial x_j}h(x+te_j)+(-1)^k\left(\frac{\partial}{\partial x_1}+\frac{\partial}{\partial x_2}+\ldots+\frac{\partial}{\partial x_n}\right)^kh(x+tv)=0
\]
Furthermore doing $t=0$ we get $T_k(h)=0$.
\end{proof}

\begin{proof}[{\sl Proof of Theorem \ref{thm:baz}}]
Now suppose that $f\colon\mathcal{U}\to f(\mathcal{U})\subset\mathbb{C}^n$
is a biholomorphism such that the vector fields 
$f_*\frac{\partial}{\partial x_1}$, $f_*\frac{\partial}{\partial x_2}$,
$\ldots$, $f_*\frac{\partial}{\partial x_n}$, $f_*\left(-\frac{\partial}{\partial x_1}-\frac{\partial}{\partial x_2}-\ldots-\frac{\partial}{\partial x_n}\right)$ satisfy the
barycentric property. Setting $f=(f_1,f_2,\ldots,f_n)$ we see
that it is equivalent to 
\[
\displaystyle\sum_{j=1}^nf_\ell(x+te_j)+f_\ell(x+tv)=(n+1)f_\ell(x)\qquad\forall\,1\leq\ell\leq n.
\]
Therefore each component $f_\ell$ of $f$ satisfies (\ref{eq:claim}) so 
that $f_\ell$ belongs to $\displaystyle\bigcap_{k\geq 2}\ker(T_k)$
for any $1\leq \ell\leq n$ (Lemme \ref{claim1}). The idea is 
to prove that 
$\displaystyle\bigcap_{k\geq 2}\ker(T_k)\subset\mathbb{C}[x_1,x_2,\ldots,x_n]$: 
if $h\in\displaystyle\bigcap_{k\geq 2}\ker(T_k)$, 
then $h$ is a polynomial. 

Let $\mathcal{P}$ be the Noetherian ring of linear differential 
operators on $\mathcal{O}(\mathcal{U})$ with constant coefficients
\[
\mathcal{P}=\big\{P\left(\frac{\partial}{\partial x_1},\frac{\partial}{\partial x_2},\ldots,\frac{\partial}{\partial x_n}\right)\,\vert\, P\in\mathbb{C}[z_1,z_2,\ldots,z_n]\big\}
\]
and let $\mathcal{I}=\langle T_k\,\vert\,k\geq 2\rangle$ be the ideal of 
$\mathcal{P}$ generated by all the operators $T_k$, 
$k\geq 2$. Note that if $S$ belongs to $\mathcal{I}$, 
then $\displaystyle\bigcap_{k\geq 2}\ker(T_k)$ is 
contained in $\ker(S)$. 

\begin{claim}\label{claim2}
{\sl There exists $p\in\mathbb{N}$ such that 
$\frac{\partial^p}{\partial x_j^p}$ belongs to 
$\mathcal{I}$ for all $1\leq j\leq n$.}
\end{claim}

Claim \ref{claim2} implies that if $h$ belongs
to $\displaystyle\bigcap_{k\geq 2}\ker(T_k)$, then $h$ is a 
polynomial of degree at most $n(p-1)$.

\begin{proof}[{\sl Proof of Claim \ref{claim2}}]
Let $\Phi\colon\mathcal{P}\to\mathcal{O}_n$ be the 
unique ring homomorphism satisfying 
\[
\Phi\left(\frac{\partial}{\partial x_j}\right)=z_j  \qquad\,\forall 1\leq j\leq n.
\]
Note that $\Phi(T_k)=z_1^k+z_2^k+\ldots+z_n^k+(-1)^k(z_1+z_2+\ldots+z_n)^k$.
Let us set
\begin{align*}
&P_k(z)=z_1^k+z_2^k+\ldots+z_n^k+(-1)^k(z_1+z_2+\ldots+z_n)^k,
&&\widetilde{\mathcal{I}}=\langle P_k\,\vert\,k\geq 2\rangle,
&&\Phi(\mathcal{I})=\widetilde{\mathcal{I}}.
\end{align*}

\begin{claim}\label{claim3}
{\sl One has 
\[
Z(\widetilde{\mathcal{I}})=\big\{z\in\mathbb{C}^n\,\vert\,P_k(z)=0\quad\forall\,k\geq 2\big\}=\big\{0\big\}.
\]}
\end{claim}
From $Z(\widetilde{\mathcal{I}})=\{0\}=Z(\mathfrak{m}_n)$ one gets (using 
the definition of $\sqrt{\widetilde{\mathcal{I}}}$) that 
$\sqrt{\widetilde{\mathcal{I}}}=\mathfrak{m}_n$. Accor\-ding to 
Hilbert's theorem (Nullstellensatz) one obtains that 
$\widetilde{\mathcal{I}}\supset \mathfrak{m}_n^p$ for some $p$. 
As a result $z_j^p$ belongs to $\widetilde{\mathcal{I}}$
for all $1\leq j\leq n$ and so $\frac{\partial^p}{\partial z_j^p}$
belongs to $\mathcal{I}$ for all $1\leq j\leq n$.
\end{proof}

\begin{proof}[{\sl Proof of Claim \ref{claim3}}]
Define $S:=-(z_1+z_2+\ldots+z_n)$ so that $P_k=z_1^k+z_2^k+\ldots+z_n^k+S^k$.
Therefore if $z$ belongs $Z(\widetilde{\mathcal{I}})$, then
\[
(**)\left\{
\begin{array}{lllll}
z_1+z_2+\ldots+z_n+S=0\\
z_1^2+z_2^2+\ldots+z_n^2+S^2=0\\
\ldots \\
z_1^n+z_2^n+\ldots+z_n^n+S^n=0\\
z_1^{n+1}+z_2^{n+1}+\ldots+z_n^{n+1}+S^{n+1}=0
\end{array}
\right.
\]
Doing $S=z_{n+1}$ system $(**)$ is equivalent to
$Q_{n+1}v^{\,\,\!t}=0$ where $Q_{n+1}$ is the 
matrix
\[
Q_{n+1}(z)=\left(\begin{array}{cccc}
z_1 & z_2 & \ldots & z_{n+1} \\
z_1^2 & z_2^2 & \ldots & z_{n+1}^2 \\
\vdots & & & \vdots \\
z_1^{n+1} & z_2^{n+1} & \ldots & z_{n+1}^{n+1} \\
\end{array}
\right)
\]
and $v=(1,1,\ldots,1)$. Finally it can be checked by induction
on $n\geq 0$ that if $Q_{n+1}(z)u^{\,\,\!t}=0$ for 
some $u=(u_1,u_2,\ldots,u_{n+1})$, where $u_j>0$
for all $1\leq j\leq n+1$, then $z=0$.
\end{proof}
\end{proof}

\section{Description of the $2$-chambars}\label{sec:2cham}

\subsection{Examples coming from foliations by straight lines}\label{subsec:foliationsbystraightlines}

In order to precise the previous statements we recall 
the classification of foliations by straight lines on 
$\mathbb{P}^3_\mathbb{C}$ that can be found in 
\cite{Cerveau} (according to Jorge Pereira  
this classification was already known to 
Kummer). We do not know if such a classification
exists on $\mathbb{P}^3_\mathbb{R}$.

Let $\mathcal{F}$ be a holomorphic foliation on 
$\mathbb{P}^n_\mathbb{C}$. Chow theorem asserts
that $\mathcal{F}$ is algebraic; such a 
foliation $\mathcal{F}$ has singularities. We 
say that $\mathcal{F}$ is a \textsl{foliation by straight lines}
if the generic leaf is contained in a line 
(in fact a line without a few points). Let us 
mention the difference between the real case: 
foliations by straight lines of $\mathbb{P}^3_\mathbb{R}$
without singularities exist. The typical 
example is produced by Hopf fibration: the real
projectivization of complex vector lines 
of~$\mathbb{C}^2\simeq\mathbb{R}^4$ gives such 
a foliation $\mathcal{H}$. Setting $z=x_1+\mathbf{i}x_2$
and $w=x_3+\mathbf{i}x_4$ these foliations have 
the first integral
\[
\frac{z}{w}=\frac{z\overline{w}}{\vert w\vert^2}=\frac{x_1x_3-x_2x_4+\mathbf{i}(x_1x_4+x_2x_3)}{x_3^2+x_4^2}
\]
In particular $\frac{x_1x_3-x_2x_4}{x_3^2+x_4^2}$ and 
$\frac{x_1x_4+x_2x_3}{x_3^2+x_4^2}$ are
real first integrals of $\mathcal{H}$.

Let us recall the classification of foliations by 
straight lines of $\mathbb{P}^3_\mathbb{C}$:

\begin{thm}[\cite{Cerveau}]\label{thm:clasfolP3}
{\sl Every holomorphic foliation by straight lines in 
$\mathbb{P}^3_\mathbb{C}$ is, up to linear 
equivalence, of one of the following types
\begin{itemize}
\item[1.] a radial foliation at a point,

\item[2.] a radial foliation "in the pages 
of an open book", {\it i.e.} a family of 
radial foliations of dimension $2$ each 
contained in a plane of the family of planes
containing a fixed line;

\item[3.] a foliation associated with the twisted
cubic $t\mapsto(t,t^2,t^3)$; here the $($closure of the$)$ leaves of the
foliation are the chords and the lines tangent to 
the twisted cubic.
\end{itemize}}
\end{thm}

Foliations of the first type correspond to foliations
by parallel lines in a well-chosen affine chart
(singular point at infinity).

To construct a foliation of the second type we 
consider an open book, {\it i.e.} a pencil of 
hyperplanes, for instance $\frac{x_1}{x_2}=$ constant;
in any page $\frac{x_1}{x_2}=c$ we fix a point 
$(\underline{x_1},c\underline{x_2},\underline{x_3})$
and ask that any leaf of $\mathcal{F}$ is a line 
contained in a page $\frac{x_1}{x_2}=c$ and 
passes through the prescribed point 
$(c\underline{x_2},\underline{x_2},\underline{x_3})$
(\emph{see} \cite{Cerveau} for further details).

\begin{align*}
    &\includegraphics[width=0.2\linewidth]{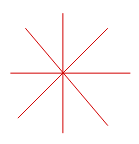} 
    &&\includegraphics[width=0.2\linewidth]{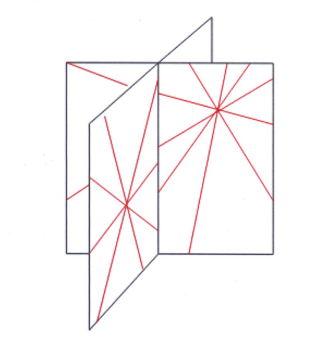}
    &&\includegraphics[width=0.2\linewidth]{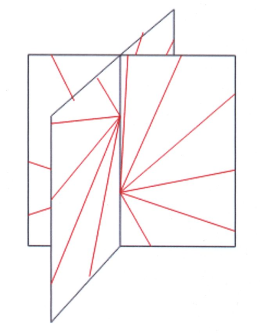}
    &&\includegraphics[width=0.2\linewidth]{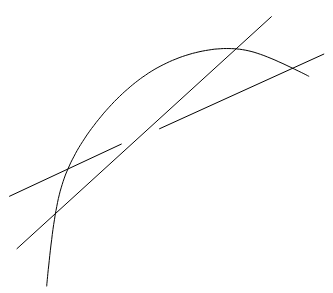} \\
    & \quad\quad\text{type 1.} && \quad \quad\text{type 2.} &&  \quad\quad\text{type 2.} &&  \quad\quad\text{type 3.}
\end{align*}

Remark that Theorem \ref{thm:clasfolP3}  gives the description 
of algebraic foliations by straight lines in the affine 
space~$\mathbb{C}^3$.

Let us now explain how we can construct a 
$2$-chambar from a foliation $\mathcal{F}$ by 
lines defined on an open subset $\mathcal{U}$
of $\mathbb{C}^n$. 
For a good choice of the affine 
coordinates $x_i$ the foliation
$\mathcal{F}$ is defined by a 
vector field 
\[
X=\frac{\partial}{\partial x_1}+\alpha_2\frac{\partial}{\partial x_2}+\alpha_3\frac{\partial}{\partial x_3}+\ldots+\alpha_n\frac{\partial}{\partial x_n}
\]
on $\mathcal{U}$. 
Of course, in general, the $\alpha_i$'s are 
meromorphic and we consider 
$\mathcal{U}^*=\displaystyle\mathcal{U}\smallsetminus\bigcup_{i=2}^n(\text{poles of $\alpha_i$})$.
Then if $m$ belongs to $\mathcal{U}^*$
the trajectory of $X$ passing through $m$
is a line $D_m$ and $\exp(tX)_{\vert D_m}$
is a translation flow on $D_m$. The 
pair $(X,-X)$ thus defines a $2$-chambar.

One can next consider $f\cdot X$, where $f$
is any meromorphic first integral of $X$, 
instead of~$X$. Since $f$ is constant on
 the trajectories of $X$, $f\cdot X$ still 
 defines a translation flow on any trajectory
 of $X$, and $(f\cdot X,-f\cdot X)$ is also a 
 $2$-chambar.

\subsection{Some properties}

The barycentric property for a 
$2$-chambar $\mathrm{Ch}(X_1,X_2)$
implies that $X_1+X_2=0$ and can be 
rewritten as 
\[
\varphi_t(x)+\varphi_{-t}(x)=2x\qquad\forall\,x\in\mathcal{U}
\]
where $\varphi_t$ denotes the flow of $X=X_1$.

Differentiating the previous equality with respect to time $t$, 
we get
\[
\overset{\bullet}{\varphi_t}(x)-\overset{\bullet}{\varphi_{-t}}(x)=X\big(\varphi_t(x)\big)-X\big(\varphi_{-t}(x)\big)=0;
\]
differentiating a second time with respect to $t$, we obtain 
\[
DX\big(\varphi_t(x)\big)\overset{\bullet}{\varphi_t}(x)+DX\big(\varphi_{-t}(x)\big)\overset{\bullet}{\varphi_{-t}}(x)=0
\]
where $DX\colon\mathcal{U}\to\mathbb{R}^n$ (or 
$DX\colon\mathcal{U}\to\mathbb{C}^n$) denotes the differential 
of $X$.

If 
$X=\displaystyle\sum_{i=1}^n\alpha_i(x)\frac{\partial}{\partial x_i}$, the above relation 
is equivalent to
\[
DX(X)=\displaystyle\sum_{i=1}^nX(\alpha_i)\frac{\partial}{\partial x_i}=0
\]
In particular the coefficients $\alpha_k$ are first integrals of  
$X$, $2\leq k\leq n$. As a result the $\alpha_k$ are constant along the trajectories 
of $X$; these trajectories are thus (contained in) lines.

Note that in dimension $1$ we can write $X=\alpha\frac{\partial}{\partial x}$
and the above relation is equivalent to $\alpha\frac{\partial\alpha}{\partial x}=0$; 
hence~$\alpha$ is constant. On any of its trajectories the flow 
of $X$ thus coincides with the flow of a constant vector field.
As a result one can state:

\begin{thm}\label{thm:dim1}
{\sl Let $\mathcal{U}$ be an open subset of $\mathbb{R}^n$ 
$($resp. $\mathbb{C}^n)$.
Let $X_1$, $X_2$ be two analytic $($resp. holomorphic$)$ vector fields 
on $\mathcal{U}$. Assume that $X_1$ and $X_2$ satisfy the barycentric 
property. 

Then the leaves of $\mathcal{F}_{X_1}=\mathcal{F}_{X_2}$ are contained in 
lines; on each of these lines the flows $\exp(tX_k)_{\vert D}$
are translation flows.}
\end{thm}

In particular in dimension $1$ any $2$-chambar $(X,-X)$ is 
produced by a constant vector field.
Remark also that any local 
$2$-chambar in one variable 
can be globalized.

\begin{cor}
{\sl Let $X$ be a rational vector field on $\mathbb{C}^n$.
Assume that $(X,-X)$ defines a $2$-chambar. Then
$\exp(tX)=\mathrm{id}+tX^0$ defines a flow of 
birational maps of $\mathbb{C}^n$.}
\end{cor}

Note that in $\exp(tX)=\mathrm{id}+tX^0$ the 
letter $X^0$ denotes the map whose components are the 
components of the vector field $X$, a system
of coordinates having been chosen.

\begin{rem}
In the real case there is an other proof of Theorem 
\ref{thm:dim1} which is geometric. 

Let $\Gamma$ be a generic leaf of $\mathcal{F}_{X_1}=\mathcal{F}_{X_2}$.
Assume that $\Gamma$ is not (contained in) a line. If 
$x\in\Gamma$ is a generic point, then there exists an hyperplane 
$\Sigma$ tangent to $\Gamma$ at $x$ such that
\begin{itemize}
\item[$\diamond$]  the germ $\Gamma_{,x}$ is contained in one of the 
half spaces delimited by $\Sigma$,

\item[$\diamond$] $\Gamma_{,x}\cap\Sigma=\{x\}$
\end{itemize}

\begin{center}
\includegraphics[width=6cm]{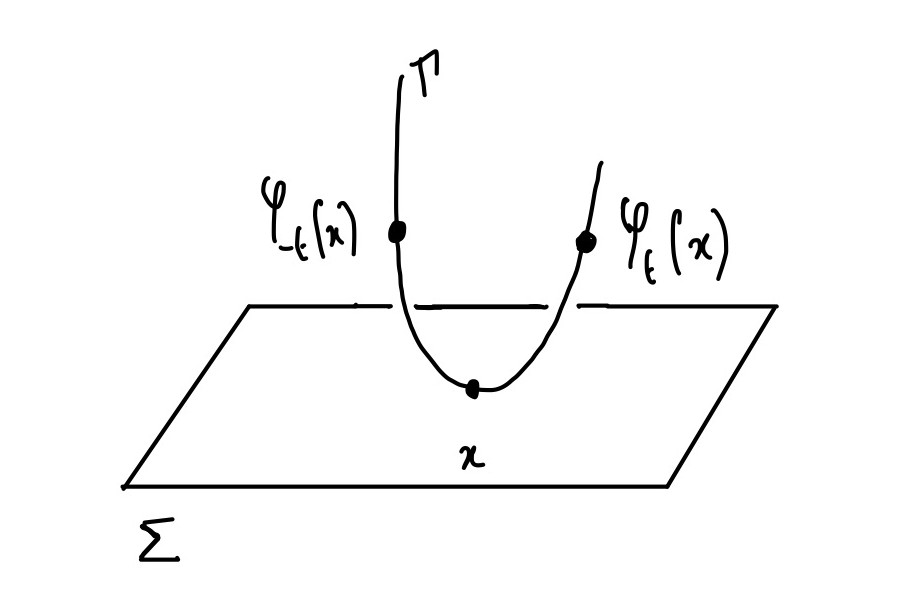}
\end{center}

If we set $\varphi_t=\exp tX_1$, 
then $\varphi_t(x)-x+\varphi_{-t}(x)-x\not\equiv 0$: contradiction.
\end{rem}

Let 
$X=\displaystyle\sum_{i=1}^n\alpha_i\frac{\partial}{\partial x_i}$
be a germ of vector fields at the origin of
$\mathbb{C}^n$. Denote by 
$\mathrm{Sing}(X)=\{\alpha_1=\alpha_2=\ldots=\alpha_n=0\}$ 
the singular set of $X$.

The following statement is a special case of 
Theorem \ref{thm:singdim}; its proof is algebraic 
in contrast with the geometric proof of 
Theorem \ref{thm:singdim}.

\begin{thm}\label{thm:2sing}
{\sl Let $\mathrm{Ch}(X,-X)$ be a $2$-chambar 
at $0\in\mathbb{C}^n$. Assume that
$X$ is singular at $0$, that is 
$\{0\}\subset\mathrm{Sing}(X)$. 

Then $\dim\mathrm{Sing}(X)\geq 1$.}
\end{thm}

\begin{proof}[{\sl Proof}]
The condition $X(\alpha_k)=0$, $1\leq k\leq n$, is equivalent to
\[
\displaystyle\sum_{i=1}^n\alpha_i\frac{\partial\alpha_k}{\partial x_i}=0\qquad 1 \leq k\leq n.
\]
Hence the partial derivatives 
$\left(\frac{\partial\alpha_k}{\partial x_1},\frac{\partial\alpha_k}{\partial x_2},\ldots,\frac{\partial\alpha_k}{\partial x_n}\right)$ 
are relations of the ideal $(\alpha_1,\alpha_2,\ldots,\alpha_n)$. 

Assume by contradiction that 
$\dim\mathrm{Sing}(X)=0$.  
Then according to \cite{Tougeron} the 
relations are gene\-rated by the trivial
relations 
\[
(0,0,\ldots,0,\underbrace{\alpha_j}_{\text{$i$th  coordinate}},0,\ldots,0,\underbrace{-\alpha_i}_{\text{ $j$th coordinate}},0,0,\ldots,0);
\]
this gives a contradiction with the following fact: the algebraic 
multiplicity at $0$ of one of  
the $\frac{\partial \alpha_k}{\partial x_i}$
is less than the algebraic multiplicity
at $0$ of $\alpha_k$.
\end{proof}

\begin{rem}
Let $u\in\mathcal{O}^*(\mathbb{C}^n,0)$ be 
a unit. Then the vector field 
$u\cdot\displaystyle\sum_{i=1}^nx_i\frac{\partial}{\partial x_i}$ 
which has linear trajectories
can not belong to a $2$-chambar; but the 
rational field 
$\frac{1}{x_1}\displaystyle\sum_{i=1}^nx_i\frac{\partial}{\partial x_i}$
can.
\end{rem}

\section{Rigid chambars}\label{sec:rigidcham}

\subsection{Flows which are polynomial in the time $t$}

\begin{defi}
Let $X$ be an holomorphic vector field on the 
open set $\mathcal{U}\subset\mathbb{C}^n$. 
We say that~$X$ is a
\textsl{$t$-polynomial} vector field if 
$t\mapsto \exp tX$ is polynomial. The 
\textsl{$t$-degree} of $X$ is the usual 
degree in the variable $t$ and is 
denoted by $t.d(X)\in\mathbb{N}\cup\{\infty\}$. 
\end{defi}

We have seen a lot of examples of $t$-polynomial 
vector fields: constant vector fields, 
nilpotent vector fields, the vector field 
$2\sqrt{x}\frac{\partial}{\partial x}$, ...

\medskip

If $\mathcal{U}=\mathbb{C}^n$, then the 
trajectories of a $t$-polynomial vector field
are points or rational curves. 

\begin{pro}
{\sl Let $X$ be a $t$-polynomial vector field of $t$-degree $\nu$ on 
the open set $\mathcal{U}\subset~\mathbb{C}^n$.
Write $\exp tX$ as 
$\mathrm{Id}+tF_1+t^2F_2+\ldots+t^\nu F_\nu$, 
with $F_k\in\mathcal{O}(\mathcal{U})$. Then 
the components $F_{\nu,1}$, $F_{\nu,2}$, 
$\ldots$, $F_{\nu,n}$ of $F_\nu$ are first 
integrals of $X$.

In particular in the $1$-dimensional case, 
$F_\nu$ is a non-zero constant.}
\end{pro}

\begin{proof}[{\sl Proof}]
It is a direct consequence of the 
identity $\exp tX\circ \exp sX=\exp(s+t)X$: 
the coefficient of $t^\nu$ in that identity
is exactly
\[
F_\nu(\exp sX)=F_\nu.
\]
This implies the statement.
\end{proof}

Remark the $F_{\nu,k}$ may be constant; this is the case for 
the flow of $X=2\sqrt{x}\frac{\partial}{\partial x}$. A contrario 
if a $t$-polynomial vector 
field $X$ of degree $\nu$ is 
singular at a point, say $0$ ({\it i.e.} $X(0)=0$), 
then obviously some of 
the $F_{\nu,k}=\frac{X^{\nu}(x_k)}{\nu !}$ 
are not constant. 
In particular in dimension $2$, a $t$-polynomial vector field $X$
singular at the origin $0\in\mathbb{C}^2$, $X(0)=0$,
has a non-constant holomorphic first integral $f$. 
The generic leaves of $X$ are the levels of $f$;
note that since the flow is polynomial one has the following important property: $X_{\vert f^{-1}(0)}\equiv 0$.

The $t$-polynomial vector fields produce examples
of $p$-chambars as we have seen pre\-viously.
Typically if $\sigma$ is a  primitive $\nu$-th root
of unity and $t\cdot d(X)=\nu$, then 
$X$, $\sigma X$, $\ldots$, $\sigma^{\nu-1}X$
defines a (rigid) $\nu$-chambar. 

If $t\cdot d(X)=1$, then $\exp tX=\mathrm{Id}+tF_1$
and the foliation associated to $X$ is a 
foliation by straight lines. 
Conversely to a foliation by straight lines we 
can associate a (meromorphic) $t$-polynomial vector
field $X$ such that $t\cdot d(X)=1$.

In dimension $2$, consider a foliation 
given by the vector field 
$X=f\frac{\partial}{\partial x}+\frac{\partial}{\partial y}$. 
Then $X$ is a $t$-polynomial vector field of degree
$1$ if and only if the foliation $\mathcal{F}_X$
is a foliation by straight lines; this means that $f$ 
satisfies the non-linear PDE
\[
0=X(f)=f\frac{\partial f}{\partial x}+\frac{\partial f}{\partial y};
\]
note that this PDE is the famous inviscid Burgers' equation, 
a well-known PDE in fluid mechanic. Similarly
$t$-polynomial vector fields of degree $2$ on open set of
$\mathbb{C}^2$ correspond to foliations in parabolas
etc. In that case appear generalizations of Burgers' 
equation as the reader can see.

 The following result gives the classification of the $t$-polynomial vector field on the complex line.

\begin{thm}\label{thm:classtpolyvf}
{\sl Let $X(x)=a(x)\frac{\partial}{\partial x}$ be a germ at
$0\in\mathbb{C}$ of a holomorphic vector field. 
Assume that the flow of~$X$ is polynomial in $t$ of
$t$-degree $\ell$. Then $a=f'\circ\phi$ where
\begin{itemize}
\item[$\diamond$] $f$ is a polynomial of degree $\ell$ 
with $f(0)=0$ and $f'(0)=a(0)\not=0$,; 

\item[$\diamond$] $\phi\colon(\mathbb{C},0)\to(\mathbb{C},0)$
is a local inverse of $f$: $f\circ\phi(x)=x$.
\end{itemize}

In other 
words $X$ is conjugate to the constant 
vector field $\frac{\partial}{\partial x}$ via a polynomial (local)
diffeomorphism.}
\end{thm}

\begin{proof}[{\sl Proof}]
Suppose that $a(0)\not=0$. In this case the vector field 
$X$ is conjugated to a constant vector field, say
$Y=\frac{\partial}{\partial x}$. Let $f$ be an element
of $\mathrm{Diff}(\mathbb{C},0)$ such that $f_*Y=X$. 
The flow $\varphi_t$ of $X$ can be written as
\[
\varphi_t(x)=f\big(f^{-1}(x)+t\big),
\]
where $f^{-1}\in\mathrm{Diff}(\mathbb{C},0)$ is the 
local inverse of $f$. We thus have $a(x)=f'\circ f^{-1}(x)$.
As we have seen in \S\ref{sec:remeg} (\ref{eq:al}) 
\[
\varphi_t(x)=x+\displaystyle\sum_{k\geq 1}\frac{1}{k!}X^k(x)t^k;
\]
since $t\cdot d(X)=d$ we must have $X^k(x)=0$ for 
all $k\geq d+1$. Note that the functions $f_k(x)=X^k(x)$, $k\geq 1$, 
satisfy the recurrence rule:
\begin{itemize}
\item[(i)] $f_1=a$, 

\item[(ii)] $f_{k+1}=a f'_k$, $\forall\,k\geq 1$.
\end{itemize}
Let us define another sequence of germs at $0\in\mathbb{C}$
as $g_k=f_k\circ f$, $k\geq 1$. This new sequence 
satisfies the recurrence rule:
\begin{itemize}
\item[(i')] $g_1=f_1\circ f=a\circ f=f'$, 

\item[(ii')] $g_{k+1}=f_{k+1}\circ f=a\circ f\cdot f'_k\circ f=f'_k\circ f\cdot f'=(f_k\circ f)'=g'_k$, $\forall\,k\geq 1$.
\end{itemize}
Therefore from (i') and (ii') we get for all $k\geq 1$
\[
g_k=\frac{\partial^k f}{\partial x^k}.
\]
Now, as $f_{\ell+1}\equiv 0$ we have $g_{\ell+1}\equiv 0$ and so 
$f$ is a polynomial of degree at most $\ell$. But since
the flow $\varphi_t$ has degree $\ell$, $f$ must be
of degree exactly $\ell$.

Suppose by contradiction that $a(0)=0$. In this case 
we can write $a(x)=x^\ell h(x)$ where $\ell\geq 1$
and $h(0)\not=0$. But using the recurence 
rule (ii) it is possible to prove that 
$f_k(x)=x^{\ell k-k+1}h_k(0)$ where $h_k(0)\not=0$ 
for all $k\geq 1$. As a consequence the flow can 
not be polynomial in $t$.
\end{proof}

\begin{rem}
Fixing $x=0$ in the third line of the proof we immediately 
get that $f$ is polynomial; we followed a longer process 
because it is essential in the study of the case $a(0)=0$.
\end{rem}

Theorem \ref{thm:classtpolyvf} implies that a germ 
of holomorphic $t$-polynomial vector field in one variable 
has no singularities. This is not the case in $n\geq 2$ 
variables (consider for instance $x_2\frac{\partial}{\partial x_1}$).
Nevertheless Theorem \ref{thm:classtpolyvf} has a natural
generalization in $n\geq 2$ variables, but with an 
additional assumption of "non-singularities":

\begin{thm}
{\sl Let $X=\displaystyle\sum_{i=1}^na_i(x)\frac{\partial}{\partial x_i}$
be a germ at $0$ of a non-singular $t$-polynomial 
vector field, $a_1(0)\not=0$ for fixing ideas. 

There exists 
$f\in\mathrm{Diff}(\mathbb{C}^n,0)$ a germ of diffeomorphism
which is polynomial in the variable $x_1$ such that 
$X=f_*\frac{\partial}{\partial x_1}$, {\it i.e.}
$\varphi_t(x)=f(f^{-1}(x)+te_1)$ where $\varphi_t$
is the flow of $X$ and $f^{-1}$ is the local 
inverse of $f$
at $0$.}
\end{thm}

\begin{proof}[{\sl Proof}]
Let $f$ be a local conjugacy between $X$ and 
$\frac{\partial}{\partial x_1}$
satisfying $f(0,x_2,x_3,\ldots,x_n)=(0,x_2,x_3,\ldots,x_n)$
(it is well-known that such a conjugacy exists). In 
particular $\varphi_t(x)=f(f^{-1}(x)+te_1)$ and 
\[
\varphi_t(0,x_2,x_3,\ldots,x_n)=f(t,x_2,x_3,\ldots,x_n);
\]
in particular $f$ is thus polynomial in the variable $x_1$.
\end{proof}

\subsection{Rigid chambars on $\mathbb{R}^n$ and foliations by straight lines}

The following statement generalizes to 
the real case the property satisfied by the 
$2$-chambars:

\begin{thm}\label{thm:realnotcomp}
{\sl If $\mathrm{Ch}(a_1X,a_2X,\ldots,a_pX)$
is a rigid $p$-chambar on an open subset of 
$\mathbb{R}^n$, then the foliation $\mathcal{F}_X$
associated to $X$ is a foliation by straight lines.}
\end{thm}

\begin{proof}[{\sl Proof}]
As in the proof of Theorem \ref{thm:dim1}
we get by successive derivations the 
equalities
\[
\left\{
\begin{array}{ll}
\displaystyle\sum_{k=1}^pa_k=0\\
\left(\displaystyle\sum_{k=1}^pa_k^2\right)DX\cdot X=0
\end{array}
\right.
\]
Since $a_k\not=0$ for any $1\leq k\leq p$
one has $DX\cdot X=0$. As a 
result all the non-singular trajectories
of $X$ are straight lines.
\end{proof}

Theorem \ref{thm:realnotcomp} can not be generalized 
to the complex case. Let us give a 
counter example of Theorem \ref{thm:realnotcomp}
in the complex case in dimension $2$.
Consider on $\mathbb{C}^2$ the linear
vector field
\[
X=x\frac{\partial}{\partial x}+2y\frac{\partial}{\partial y}.
\]
The closure of its trajectories are 
the parabola $y=cx^2$ with
$c\in\mathbb{P}^1_\mathbb{C}$ 
(if $c\in\{0,\,\infty\}$, then the 
trajectory is a line). Let us 
consider the vector field
\[
Y=\frac{1}{x}X=\frac{\partial}{\partial x}+\frac{2y}{x}\frac{\partial}{\partial y}
\]
which is holomorphic outside $x=0$. Its 
$1$-parameter group is the group of 
birational maps 
\[
(\exp tY)(x,y)=\left(x+t,\left(\frac{x+t}{x}\right)^2y\right).
\]
Hence if $a_k$ belongs to $\mathbb{C}^*$, 
then one has 
\[
(\exp ta_kY)(x,y)=\left(x+a_kt,\left(\frac{x+a_kt}{x}\right)^2y\right).
\]
Take some non zero constants $a_1$, $a_2$,
$\ldots$, $a_p$, $p\geq 3$, such that 
\[
\displaystyle\sum_{k=1}^pa_k=\displaystyle\sum_{k=1}^pa_k^2=0.
\]
Then the vector fields $Y_k=a_kY$, 
$1\leq k\leq p$, form a $p$-chambar
on the open set
$\mathcal{U}=\mathbb{C}^2\smallsetminus\{x=~0\}$.
But the trajectories of $Y$, that 
are almost the trajectories of $X$, 
are not straight lines.

\begin{rem}
Let $X$ be a germ at $0\in\mathbb{C}^n$
of holomorphic vector 
field. Suppose that there exist some constants $a_1$, $a_2$, 
$\ldots$, $a_p$ such 
that the $X_k=a_kX$ generate a $p$-chambar.
If $X$ is not singular at $0$, $X(0)\not=0$,
then $\mathrm{Ch}(a_1X,a_2X,\ldots,a_pX)$
is locally conjugate to the constant 
$p$-chambar
$\mathrm{Ch}\left(a_1\frac{\partial}{\partial x},a_2\frac{\partial}{\partial x},\ldots,a_p\frac{\partial}{\partial x}\right)$. 
Indeed if $\phi$ is a local diffeomorphism
that conjugates~$X$ to 
$\frac{\partial}{\partial x}$ and if 
$a$ belongs to $\mathbb{C}$, then 
$\phi$ conjugates $aX$ to 
$a\frac{\partial}{\partial x}$. Be careful it 
does not mean that the image of a constant 
$p$-chambar via a diffeomorphism 
is a $p$-chambar (see Theorem 
\ref{thm:cstpchamb}).
\end{rem}

\subsection{Rigid and semi-rigid chambars on $\mathbb{C}^n$}

\subsubsection{Rigid chambars on $\mathbb{C}^n$ and 
$t$-polynomial vector fields}

\begin{thm}\label{thm:rigidp}
{\sl Let $\mathrm{Ch}(X,a_1X,a_2X,\ldots,a_{p-1}X)$ be 
a germ at $0\in\mathbb{C}^n$ of rigid $p$-chambar. 

Then the flow $\varphi_t$ of $X$ is polynomial of
degree at most $p-1$, as a function of the time $t$.

If $t.d(X)=d$, then $a_1$, $a_2$, $\ldots$, $a_p$
satisfy
\[
a_1^\ell+a_2^\ell+\ldots+a_p^\ell=0 \quad\forall\,1\leq\ell\leq d.
\]

In particular if $d=p-1$, then $a_1^p=a_2^p=\ldots=a_p^p$. 

Moreover
if the $p$-chambar $(a_1X,a_2X,\ldots,a_pX)$ is irreducible,
then $\frac{a_k}{a_1}$ is a primitive $p$-th root of unity
for some $1\leq k\leq p$.}
\end{thm}

\begin{proof}[{\sl Proof}]
Write $X$ as $\displaystyle\sum_{k=1}^n X_k\frac{\partial}{\partial x_k}$;
the barycentric condition is the following
\begin{eqnarray*}
& & px_j=px_j+t\big(a_1+a_2+\ldots+a_p\big)X_j+\frac{t^2}{2}\big(a_1^2+a_2^2+\ldots+a_p^2\big)X(X_j)\\
& & \hspace{1cm}+\ldots+\frac{t^k}{k!}\big(a_1^k+a_2^k+\ldots+a_p^k\big)X^{k-1}(X_j)+\ldots
\end{eqnarray*}
for $j=1$, $2$, $\ldots$, $n$.

The fact that the coefficients $a_k$ are different
from zero implies that a Newton formula 
\[
a_1^\ell+a_2^\ell+\ldots+a_p^\ell
\]
is non zero for an $\ell\leq p$. As a consequence 
$X^m(X_j)\equiv 0$ for all $m\geq \ell-1$ and $1\leq j\leq n$. 
This implies that the flow of $X$, and 
the flows of the $a_kX$, are polynomial in $t$.

The other facts can be checked by the reader.
\end{proof}

\subsubsection{A property of the singular set}

Let $X$ be a holomorphic vector field 
defined on an open subset $\mathcal{U}$ of $\mathbb{C}^n$. 
Denote by $\mathcal{F}_X$ the singular one 
dimensional foliation defined by $X$ on 
$\mathcal{U}$. 
A \textsl{separatrix} $\gamma$ of $X$ through 
$x_0\in\mathrm{Sing}(X)$ is a germ of analytic 
curve at $x_0$ such that 
\begin{itemize}
\item[$\diamond$] $X\not\equiv 0$ on 
$\gamma\smallsetminus\{x_0\}$,

\item[$\diamond$] $x_0$ belongs to $\gamma$,

\item[$\diamond$] $\gamma\smallsetminus\{x_0\}$ is a leaf of the
germ of $\mathcal{F}_X$ at $x_0$. 
\end{itemize}
This means
that $x_0$ belongs to $\gamma$ and if 
$x$ belongs to $\gamma\smallsetminus\{x_0\}$, 
then $X(x)\not=0$ and $T_x\gamma=\mathbb{C}\cdot X(x)$.

\smallskip

Let $X$ be an holomorphic vector field defined on a 
closed ball $B=\overline{B(0,r)}$ with $X(0)=0$. 
We suppose that $X$ is a $t$-polynomial vector field, 
that is $t\mapsto\varphi_t(x)$ is polynomial in $t$,
$x\in B$, $\varphi_t=\exp tX$. Note that for 
any $x\in B$, $t\mapsto\varphi_t(x)$ can be 
extended on all the line $\mathbb{C}$. As a 
consequence if $x\in B$, the leaf $\mathcal{L}_x$
of $\mathcal{F}_X$ in $B$ is 
\begin{itemize}
\item either the 
point $x$ (case $x\in\mathrm{Sing}(X)$),

\item or the connected component of $\mathcal{L}'_x\cap B$ 
containing $x$ where $\mathcal{L}'_x$ is the 
rational curve image of $t\mapsto\varphi_t(x)$.
\end{itemize} 

\begin{lem}\label{lem:tecnotclos}
{\sl Suppose that $x$ does not belong to $\mathrm{Sing}(X)$; then
$0$ does not belong to the closure $\overline{\mathcal{L}_x}$ 
of $\mathcal{L}_x$ in~$B$.}
\end{lem}

\begin{proof}[{\sl Proof}]
Assume by contradiction that 
$0$ belongs to $\overline{\mathcal{L}_x}$. Then
there is a sequence $(t_n)_n$ of complex numbers such that
$\displaystyle\lim_{n\to +\infty}\varphi_{t_n}(x)=0$. 
Since $0\in\mathrm{Sing}(X)$ one has
$\displaystyle\lim_{n\to +\infty}\vert t_n\vert=+\infty$,
and as $t\mapsto\varphi_t(x)$ is 
polynomial (non constant)
$\displaystyle\lim_{n\to +\infty}\vert\varphi_{t_n}(x)\vert=+\infty$:
contradiction.
\end{proof}

\begin{thm}\label{thm:singdim}
{\sl Let $X\in\chi(\mathbb{C}^n,0)$ be a germ of a 
$t$-polynomial vector field at the origin of 
$\mathbb{C}^n$. 

Assume that $\mathrm{Sing}(X)\not=\emptyset$. Then 
$\dim\mathrm{Sing}(X)\geq 1$. 

Moreover $X$ has no separatrices
through a singularity.}
\end{thm}

\begin{proof}[{\sl Proof}]
Assume that $X$ is defined on the ball 
$B=\overline{B(0,r)}$ and that $0$ is an
isolated singularity of $X$. Let $(x_n)_n$ 
be a sequence of points of $B$ such that 
$\displaystyle\lim_{n\to+\infty}x_n=0$. 
The leaf $\mathcal{L}_{x_n}$ is closed
in $B$ and cuts the sphere 
$S(0,r)=B\smallsetminus B(0,r)$.
Let $y_n$ be a point in $\mathcal{L}_{x_n}\cap S(0,r)$
and $y_0$ a limit point of $y_n$, 
up to extraction $y_0=\displaystyle\lim_{n\to +\infty} y_n$. 
According to Lemma \ref{lem:tecnotclos} the point $0$ 
does not belong to $\overline{\mathcal{L}_{y_0}}$
and $\mathcal{L}_{y_0}$ can be seen as the 
leaf of the restriction of 
$\mathcal{F}_{X\vert B\smallsetminus B(0,r')}$
for $r'$ sufficiently small. The fact that 
$y_0=\displaystyle\lim_{n\to +\infty} y_n$ implies that $\mathcal{L}_{y_n}$ is 
contained in $B\smallsetminus B(0,r')$ for $n$
sufficiently large:
contradiction with $\displaystyle\lim_{n\to +\infty}x_n=0$.
\end{proof}

\begin{cor}\label{cor:rigidp}
{\sl Let $\mathrm{Ch}(a_1X,a_2X,\ldots,a_pX)$ be 
a rigid $p$-chambar on an open set 
$\mathcal{U}$ of $\mathbb{C}^n$. Then 
\begin{itemize}
\item[$\diamond$] either $\mathrm{Sing}(X)=\emptyset$, 
that is $X$ is regular;

\item[$\diamond$] or $\dim\,\mathrm{Sing}(X)\geq 1$.
\end{itemize}}
\end{cor}

\begin{eg}
Let $X$ be a linear nilpotent vector field on 
$\mathbb{C}^n$. Then the flow $\exp tX$ 
is polynomial of degree $d=\mathrm{rk}\,X$.
Moreover $\dim\mathrm{Sing}(X)=n-d$.
For instance if $X^{n-1}\not=0$, then
$\dim\mathrm{Sing}(X)=~1$.
\end{eg}

\begin{prob}
Does there exist a vector field with an isolated singularity 
belonging to a $p$-chambar?
\end{prob}

\begin{rem}
Recall that the Camacho-Sad theorem (\cite{CamachoSad}) says
that a holomorphic foliation $\mathcal{G}$ by curves at the 
origin $0$ of $\mathbb{C}^2$ has an invariant curve passing through 
$0$. As a consequence if $X$ is a $t$-polynomial vector field
at the origin $0$ of $\mathbb{C}^2$, with $X(0)=0$, then the invariant
curves of the foliation associated to $X$ are contained in 
the singular set $\mathrm{Sing}(X)$. 
\end{rem}

The previous considerations suggest 
in dimension $\geq 3$ the following 
question:

\begin{que}
Let $X$ be a germ at 
$0\in\mathbb{C}^n$ of holomorphic vector field. Assume that the closure
of the integral curves are analytic. 
Does $X$ preserve an invariant curve 
passing through $0$ ?
\end{que}

\subsubsection{Semi-rigid chambars on $\mathbb{C}^n$}

\begin{defi}
A $p$-chambar $\mathrm{Ch}(X_1,X_2,\ldots,X_p)$
on an open subset of $\mathbb{C}^n$ is 
\textsl{semi-rigid} if the $X_k$ are coli\-nears, 
that is if $X_1\wedge X_k=0$ for any $2\leq k\leq p$. 
\end{defi}

In dimension $1$ all chambars are semi-rigid.

\begin{eg}
The $3$-chambar 
$\mathrm{Ch}\left(\frac{\partial}{\partial x},y\frac{\partial}{\partial x},-(y+1)\frac{\partial}{\partial x}\right)$ 
on $\mathbb{C}^2$ is semi-rigid but not rigid.
\end{eg}

\begin{eg}
The $4$-chambar 
$\mathrm{Ch}\left(\frac{\partial}{\partial x},-\frac{\partial}{\partial x},y\frac{\partial}{\partial x},-y\frac{\partial}{\partial x}\right)$ 
on $\mathbb{C}^2$
is semi-rigid but not rigid. Note that 
it is a non-irreducible chambar.
\end{eg}

\begin{pro}\label{pro:semirigid}
{\sl Let $\mathrm{Ch}(X_1,X_2,X_3)$ be a semi-rigid $3$-chambar on 
an open subset of $\mathbb{C}^n$. 
Then one of the following holds\footnote{Note that the two properties are not mutually exclusive.}:
\begin{itemize}
\item[$\diamond$] $\mathcal{F}_{X_1}=\mathcal{F}_{X_2}=\mathcal{F}_{X_3}$
and $\mathcal{F}_{X_i}$ is a foliation by straight lines;

\item[$\diamond$] $\mathrm{Ch}(X_1,X_2,X_3)$ is a rigid 
chambar.
\end{itemize}}
\end{pro}

\begin{proof}[{\sl Proof}]
Let $\mathcal{U}$ be an open subset of $\mathbb{C}^n$ where the 
$X_i$'s are defined. Set $X_1=X$; then $X_2=fX$ 
where $f$ denotes a meromorphic function defined 
on $\mathcal{U}$. The barycentric condition implies
that $X_3=-(1+f)X$. The equality 
\[
\displaystyle\sum_{k=1}^3 DX_k\cdot X_k=0
\]
obtained by derivation from the barycentric 
property can be rewritten as follows
\begin{equation}\label{eq:semirigid}
2(1+f+f^2)DX\cdot X+(1+2f)X(f)\cdot X=0.
\end{equation}
that implies that 
\[
(1+f+f^2)X\wedge DX\cdot X=0.
\]
If $1+f+f^2=0$, then $f$ is constant and 
$\mathrm{Ch}(X_1,X_2,X_3)$ is rigid.
Otherwise, we have $X\wedge DX\cdot X=0$
and so $\mathcal{F}_X$ is a foliation
by lines.
\end{proof}

\begin{que}
Does there exist a generalization of Proposition \ref{pro:semirigid} for 
$p$-chambars, $p\geq 3$~?

The answer is positive in the real case:

\begin{pro}
{\sl Let $\mathrm{Ch}(X_1,X_2,\ldots,X_p)$ be a semi-rigid 
$p$-chambar on an open subset $\mathcal{U}\subset\mathbb{R}^n$, 
$n\geq 2$. Then 
$\mathcal{F}_{X_1}=\mathcal{F}_{X_2}=\ldots=\mathcal{F}_{X_p}$
is a foliation by straight lines.}
\end{pro}

\begin{proof}[{\sl Proof}]
Since the chambar is semi-rigid we can write 
$X_j=f_j\cdot X$ where $X$ is a vector field
on $\mathcal{U}$ and $f_j\colon\mathcal{U}\to\mathbb{R}$, 
$1\leq j\leq p$. Note that 
\[
DX_j\cdot X_j=D(f_j\cdot X)\cdot(f_jX)=f_j\cdot X(f_j)\cdot X+f_j^2\cdot DX\cdot X.
\]
In particular we get
\[
0=\displaystyle\sum_{k=1}^pDX_k\cdot X_k=\left(\displaystyle\sum_{k=1}^pf_k\cdot X(f_k)\right)\cdot X+\left(\displaystyle\sum_{k=1}^pf_k^2\right)\cdot DX\cdot X.
\]
Taking the wedge product with $X$ in the above
relation, we get
\[
\left(\displaystyle\sum_{k=1}^pf_k^2\right)X\wedge DX\cdot X=0.
\]
Since the $f_k$'s are non identically zero, 
we get $X\wedge DX\cdot X\equiv 0$. 
Therefore, $\mathcal{F}_X$ is a foliation
by straight lines.
\end{proof}

\end{que}

\section{Description of $3$-chambars and $4$-chambars in one variable}\label{sec:3cham4cham}

\subsection{Description of $3$-chambars in one variable}

\begin{thm}\label{thm:loc3bfc}
{\sl Let $\mathcal{B}$ be a holomorphic
$3$-chambar on some connected open subset of $\mathbb{C}$. Then 
\begin{itemize}
\item[$\diamond$] either $\mathcal{B}$ is a 
constant $3$-chambar;

\item[$\diamond$] or 
$\mathcal{B}=\mathrm{Ch}\Big(a(x)\frac{\partial}{\partial x},\mathbf{j}a(x)\frac{\partial}{\partial x},\mathbf{j}^2a(x)\frac{\partial}{\partial x}\Big)$, where $a(x)=\sqrt{\lambda x+\mu}$ with $\lambda\in\mathbb{C}^*$, 
$\mu\in\mathbb{C}$.
\end{itemize}

\noindent In particular, $\mathcal{B}$ is a rigid chambar.}
\end{thm}

\begin{rem}
In a certain sense Theorem 
\ref{thm:loc3bfc} shows that the set of 
$3$-chambars on a connected set
of $\mathbb{C}$ has two 
"irreducible components".
\end{rem}

\begin{proof}
Set $\mathcal{B}=\mathrm{Ch}(X_1,X_2,X_3)$.
We can write $X_k=a_k(x)\frac{\partial}{\partial x}$, 
where $a_k\in\mathcal{O}_1$, $1\leq k\leq 3$.
The barycentric property implies that 
$\displaystyle\sum_{i=1}^3X_i^k(x)=0$ for any $k\geq 1$.

Assume that the $a_i$'s are non-constant and that $X_i^2(x)\not=0$
for any $1\leq i\leq 3$. Furthermore,
$X_i^{k+1}(x)=a_i\big(X_i^k(x)\big)^\prime$ thus
\begin{equation}\label{eq:5.4}
\left\{
\begin{array}{lll}
a_1^\prime+a_2^\prime+a_3^\prime=0\\
a_1a_1^\prime+a_2a_2^\prime+a_3a_3^\prime=0\\
\big(X_1^k(x)\big)^\prime+\big(X_2^k(x)\big)^\prime+\big(X_3^k(x)\big)^\prime=0\\
a_1\big(X_1^k(x)\big)^\prime+a_2\big(X_2^k(x)\big)^\prime+a_3\big(X_3^k(x)\big)^\prime=0
\end{array}
\right.
\end{equation}
As a consequence, for any $k\geq 2$, there exists a meromorphic function $f_k$
such that $\big(X_i^k(x)\big)^\prime=f_ka_i^\prime$ for any $1\leq i\leq 3$
where $f_2\not=0$.
This yields to 
\[
X_i^{k+1}(x)=a_i\big(X_i^k(x)\big)^\prime=f_ka_ia_i^\prime=f_kX_i^2(x)\quad\forall\,1\leq i\leq 3,\,\forall\,k\geq 2
\]
and to
\[
f_k\big(X_i^2(x)\big)^\prime+f_k^\prime X_i^2(x)=\big(X_i^{k+1}(x)\big)^\prime=f_{k+1}a_i^\prime\quad\forall 1\leq i\leq 3,\,\forall\,k\geq 2.
\]
In particular 
\begin{align*}
    & f_2\big(X_i^2(x)\big)^\prime+f_2^\prime X_i^2(x)=f_3a_i^\prime, && f_k\big(X_i^2(x)\big)^\prime+f_k^\prime X_i^2(x)=f_{k+1}a_i^\prime\quad\forall\,k\geq 3
\end{align*}
and so $X_i^2(x)$ satisfies an equation of the form  $F_k\big(X_i^2(x)\big)^\prime+G_kX_i^2(x)=0$ where $F_k=f_2f_{k+1}-f_3f_k$
and $G_k=f_2^\prime f_{k+1}-f_3f_k^\prime$ $\forall\,k\geq 3$.

\begin{enumerate}
\item Let us assume first that $F_k\not=0$ for some $k\geq 3$.
In this case, for any $1\leq i\leq 3$ there exists  constants $c_i$ such that
$X_i^2(x)=c_i^2H$ where $H=\exp\big(-\int G_k/F_k\,\mathrm{d}x\big)$. 
As a result, the equa\-lity
$a_ia_i^\prime=c_i^2H$ holds for any $1\leq i\leq 3$, and $a_i^2=c_i^2K+d_i$ for some complex 
numbers $d_i$. At a generic point $x_0$
the function $K$ is holomorphic and 
(by implicit function theorem) conjugated
to $\varepsilon+x$, $\varepsilon=K(x_0)$.
The barycentric property 
\[
\displaystyle\sum_{i=1}^3a_i=\displaystyle\sum_{i=1}^3c_i\left(K+\frac{d_i}{c_i^2}\right)^{1/2}=0
\]
implies 
\[
\displaystyle\sum_{i=1}^3c_i\left(x+\varepsilon+\frac{d_i}{c_i^2}\right)^{1/2}=0
\]
which is a global identity between 
multivaluate elementary functions. 
By looking the special points 
roots of $x+\varepsilon+\frac{d_i}{c_i^2}$
we see that 
\[
\frac{d_1}{c_1^2}=\frac{d_2}{c_2^2}=\frac{d_3}{c_3^2}:=\mu.
\]
As a result,  $a_i=c_i(K+~\mu)^{1/2}$
which implies $a_ia_i^\prime=c_i^2K^\prime$.
According to the second equation of  (\ref{eq:5.4}) 
we have
\[
(c_1^2+c_2^2+c_3^2)K^\prime=0.
\]

\begin{itemize}
\item[$\bullet$] If $c_1^2+c_2^2+c_3^2\not=0$, then
$K$ is constant.

\item[$\bullet$] If $c_1^2+c_2^2+c_3^2=0$, then
up to multiplication by a constant either 
$(c_1,c_2,c_3)=(1,\mathbf{j},\mathbf{j}^2)$, or
$(c_1,c_2,c_3)=(1,\mathbf{j}^2,\mathbf{j})$.
\end{itemize}

Let us recall that if $X=b(x)\frac{\partial}{\partial x}$, then by 
formula  (\ref{eq:al}):
\begin{eqnarray}\label{eq:5.1}
(\exp tX)(x)&=&x+tb(x)+\frac{t^2}{2}b(x)b^\prime(x)+\frac{t^3}{3!}\Big(b(x)b^\prime(x)^2+b^2(x)b^{\prime\prime}(x)\Big)\nonumber\\
&&\hspace{1cm}+\frac{t^4}{4!}b(x)\Big(b(x)b^\prime(x)^2+b^2(x)b^{\prime\prime}(x)\Big)'+\ldots
\end{eqnarray}
From (\ref{eq:5.1}) we get
$\Big(\displaystyle\sum_{k=1}^3c_k^3\Big)\cdot\big(K^{\prime\prime}K^2+K'^2K\big)=0$
and 
\[
K^{\prime\prime}K^2+K'^2K=0
\]
since $c_1^3+c_2^3+c_3^3=3$. Therefore
\[
0=K^{\prime\prime}K^2+K'^2K=K(K^{\prime\prime}K+K'^2)=K(KK^\prime)^\prime 
\]
and 
$KK^\prime=\frac{\lambda}{2}$ for some $\lambda$ in 
$\mathbb{C}$. As a result $K^2=\lambda x+\mu$
for some $\mu\in\mathbb{C}$. 

\item If $F_k=0$ for all $k\geq 3$, but $G_\ell\not=0$ for some $\ell$, 
then $a_ia_i^\prime=0$, $1\leq i \leq 3$, so that the $a_i$'s are 
constant.

\item If $F_k=G_k=0$ for all $k\geq 3$ we get $f_2f_{k+1}=f_kf_3$
and $f_2^\prime f_{k+1}=f_k^\prime f_2$ for any $k\geq 3$ which 
implies the follo-wing possibilities:
\begin{itemize}
\item[$\diamond$] If $f_3=0$, then $a_ia_i^\prime=k_i$, where the 
$k_i$'s are constant. Hence, we get $a_i(x)=(k_ix+m_i)^{1/2}$. 
From the first equation of (\ref{eq:5.4}) we get 
$\displaystyle\sum_{i=1}^3(k_ix+m_i)^{1/2}=0$.
Therefore, $\frac{m_1}{k_1}=\frac{m_2}{k_2}=\frac{m_3}{k_3}=m$ and $a_i=k_i^{1/2}(x+m)^{1/2}$.

\item[$\diamond$] If $f_3\not=0$, then $f_2f_k^\prime=f_2^\prime f_k$
for any $k\geq 3$, and there exists thus a constant 
$c_k$ such that $f_k=c_kf_2$.

In particular 
\[
\big(X_i^k(x)\big)^\prime=c_kf_2a_i^\prime=c_k\big(X_i^2(x)\big)^\prime
\]
so $X_i^{k+1}(x)=c_kX_i^3(x)$ for any $k\geq 3$ and $1\leq i\leq 3$. 
But, $\big(X_i^3(x)\big)^\prime=c_3\big(X_i^2(x)\big)^\prime$ 
implies $f_2X_i^2(x)=X_i^3(x)=c_3X_i^2(x)+d_3$ where $d_3$ denotes 
a complex number. From $(f_2-c_3)X_i^2(x)=d_3$, we get $d_3=0$ and 
$f_2=c_3$; as a result, for any $k\geq 2$ there exists $\alpha_k\in\mathbb{C}$
such that 
\[
\big(X^k_i(x)\big)^\prime=f_ka_i^\prime=c_kc_3a_i^\prime=\alpha_k a_i^\prime
\quad\forall 1\leq i\leq 3.
\]
Consequently, for any $k\geq 2$ there exist $\alpha_k$ and 
$\beta_k$ in $\mathbb{C}$ such that 
\[
X_i^k(x)=\alpha_k a_i(x)+\beta_k\quad\forall\, 1\leq i\leq 3
\]
and
\[
\varphi_t^i(x)=x+F(t)a_i(x)+G(t)
\]
where $\varphi_t^i$ is the flow of $X_i$ and 
\begin{align*}
&F(t)=t+\displaystyle\sum_{k\geq 2}\frac{\alpha_k}{k!}t^k, && G(t)=\displaystyle\sum_{k\geq 2}\frac{\beta_k}{k!}t^k.
\end{align*}
Recall that if $X$ is a vector field and if 
$\varphi_t$ is its flow, then the 
derivation of an holomorphic function $f$
by $X$ satisfies 
\[
X(f)=\frac{\partial}{\partial s}f\circ\varphi_s\Big\vert_{s=0}
\]
In particular in one variable, by taking 
$f(x)=\varphi_t^i(x)$ (holomorphic function
with parameter $t$) we get 
\begin{eqnarray*}
F^\prime(t)a_i(x)+G^\prime(t)&=&a_i(x)\frac{\partial}{\partial x}\big(x+F(t)a_i(x)+G(t)\big)\\
&=&a_i(x)+F(t)a_i(x)a_i^\prime(x).
\end{eqnarray*}
Looking at the coefficients of $t$ in both sides of the equality we get
$2\alpha_2 a(x)=a(x)a^\prime(x)t$, 
that is $2\alpha_2a(x)=a(x)a^\prime(x)$
and so $2\alpha_2=a^\prime(x)$.
Therefore, 
$a_i(x)=2\alpha_2(x-x_i)$, and 
\[
\varphi_t^i(x)=x_i+e^{2\alpha_2t}(x-x_i).
\]
However, this is not possible for a chambar, unless $\alpha_2=0$, and the $a_i$'s 
are constant.
\end{itemize}
\end{enumerate}
\end{proof}

\begin{cor}\label{cor:loc3bfr}
{\sl Let $\mathcal{B}=\mathrm{Ch}(X_1,X_2,X_3)$ be a local 
$3$-chambar on $\mathbb{R}$. Then $\mathcal{B}$ is a 
constant $3$-chambar 
$\mathrm{Ch}\Big(c_1\frac{\partial}{\partial x},c_2\frac{\partial}{\partial x},c_3\frac{\partial}{\partial x}\Big)$
with $c_i$ non-zero real numbers such that  $c_1+c_2+c_3=0$.}
\end{cor}

\subsection{$p$-chambar with weights}

\begin{defi}
Let us consider $p$ 
analytic vector fields $X_1$, $X_2$, 
$\ldots$, $X_p$, defined on some open subset $\mathcal{U}$ 
of~$\mathbb{R}^n$
(resp. $\mathbb{C}^n$), with flows $t\mapsto\varphi_t^\ell$, 
$1\leq \ell\leq p$. Consider also non-zero real (resp. complex)
numbers $\alpha_1$, $\alpha_2$, $\ldots$, $\alpha_p$ and
$\alpha=\displaystyle\sum_\ell\alpha_\ell$.

\noindent We say that $X_1$, $X_2$, $\ldots$, $X_p$ define a \textsl{holomorphic $p$-chambar
with weights} $\alpha_1$, $\alpha_2$,
$\ldots$, $\alpha_p$ if 
\begin{equation}\label{eq:chambweigh}
\alpha_1\,\varphi_t^1(x)+\alpha_2\,\varphi_t^2(x)+\ldots+\alpha_p\,\varphi_t^p(x)=\alpha\,x,
\end{equation}
for all $(t,x)$ where the above formula makes sense. 
\end{defi}

\begin{rem}
This definition is equivalent to 
\[
\alpha_1\,X_1^k(x_\ell)+\alpha_2\,X_2^k(x_\ell)+\ldots+\alpha_p\,X_p^k(x_\ell)=0\quad\forall\,k\geq 1,\,\forall\,1\leq \ell\leq n.
\]
We remark that the condition is not equivalent to consider the flows
of the vector fields $\alpha_\ell X_\ell$, 
$1\leq \ell\leq n$.
\end{rem}

The classification of $3$-chambars (Theorem \ref{thm:loc3bfc}) can 
be extended to this type of chambars with an adaptation in the second
case:

\begin{thm}\label{thm:loc3bfcweight}
{\sl Assume that $X_1$, $X_2$ and $X_3$ define a holomorphic
$3$-chambar $\mathcal{B}$ with weights $\alpha_1$, $\alpha_2$, 
$\alpha_3$ on some connected open subset of $\mathbb{C}$. Then 
\begin{itemize}
\item[$\diamond$] either $\mathcal{B}$ is a 
constant $3$-chambar,

\item[$\diamond$] or 
$\mathcal{B}=\mathrm{Ch}\Big(\beta_1 a(x)\frac{\partial}{\partial x},\beta_2 a(x)\frac{\partial}{\partial x},\beta_3 a(x)\frac{\partial}{\partial x}\Big)$ where $a(x)=\sqrt{\lambda x+\mu}$ with $\lambda\in\mathbb{C}^*$, 
$\mu\in\mathbb{C}$ and
\begin{align*}
& \alpha_1\beta_1+\alpha_2\beta_2+\alpha_3\beta_3=\alpha_1\beta_1^2+\alpha_2\beta_2^2+\alpha_3\beta_3^2=0.
\end{align*}
\end{itemize}

\noindent In particular, $\mathcal{B}$ is a rigid chambar.}
\end{thm}

\subsection{Almost $p$-chambar}

\begin{defi}
Let $X$ be a vector field. We say that $X$ is 
\textsl{almost a $p$-chambar} if there exist non-zero
vector fields $X_2$, $X_3$, $\ldots$, $X_p$ such that
$(X,X_2,X_3,\ldots,X_p)$ is a $p$-chambar.

We say that $X$ is \textsl{almost a chambar}
if there exists an integer $p$ such that $X$ is almost a 
$p$-chambar.
\end{defi}

\begin{rem}
If $X$ is almost a $p$-chambar, then $X$ 
is almost a $(p+q)$-chambar for any $q\geq 2$.
\end{rem}

\begin{eg}
The constant vector fields are almost $p$-chambars
for any $p\geq 2$.
\end{eg}

\begin{eg}
Let $X$ be a nilpotent linear vector field, and 
let $p$ be its index of nilpotence.
Then $X$ is almost a $p$-chambar.
\end{eg}

We suspect that most vector fields are not almost
chambars. Let us give an explicit example 
in (real or complex) dimension $1$:

\begin{pro}
{\sl If $\lambda$ is a non-zero constant, then the vector
field $\lambda x\frac{\partial}{\partial x}$ 
is 
\begin{itemize}
\item[$\diamond$] not almost a $2$-chambar in a neighborhood of $0$;

\item[$\diamond$] not almost a $3$-chambar in a neighborhood of $0$.
\end{itemize}}
\end{pro}

\begin{rem}
The first assertion of the statement is clear. 

The second one is a direct consequence of the classification
of the $3$-chambars (Theorem~\ref{thm:loc3bfc}). 
Note that the argument does not use the property of 
nilpotency of linear chambar; indeed if $(X_1,X_2,\ldots,X_p)$ is 
a $p$-chambar containing $X=\lambda x\frac{\partial}{\partial x}$ then it 
is possible that one of the $X_k(0)$ is non zero. We conjecture that any 
semi-simple linear vector field $\displaystyle\sum_{i=1}^n\lambda_ix_i\frac{\partial}{\partial x_i}$, $\lambda_i\not=0$, 
is not almost a $p$-chambar.
\end{rem}

\subsection{Some remarks on $4$-chambars in one variable}

The $2$-chambars and $3$-chambars on an open subset
of $\mathbb{C}$ are rigid. This property is not 
satisfied by all the $4$-chambars. Consider the
vector fields $X=2\sqrt{x}\,\frac{\partial}{\partial x}$ 
and $Y=2\sqrt{x+\varepsilon}\,\frac{\partial}{\partial x}$,
$\varepsilon\not=0$, on a suitable domain of 
$\mathbb{C}$. As we know the flows of $X$ and 
$Y$ are 
\begin{align*}
& \exp tX=x+2t\sqrt{x}+t^2 && \exp tY=x+2t\sqrt{x+\varepsilon}+t^2
\end{align*}
and it is easy to see that the $4$-chambar 
$\mathrm{Ch}(X,-X,\mathbf{i}Y,-\mathbf{i}Y)$ 
is irreducible and non rigid. Such a $4$-chambar
is said to be \textsl{special}.

\begin{conjecture}
{\sl Up to affine conjugacy a $4$-chambar on an open subset
of $\mathbb{C}$ is of one of the following type:
\begin{itemize}
\item[$\diamond$] constant 
$\mathrm{Ch}\left(a_1\frac{\partial}{\partial x},a_2\frac{\partial}{\partial x},a_3\frac{\partial}{\partial x},a_4\frac{\partial}{\partial x}\right)$, $a_k\in\mathbb{C}^*$;

\item[$\diamond$] rigid of $t$-degree $2$: 
$\mathrm{Ch}(a_1X,a_2X,a_3X,a_4X)$ 
with $X=2\sqrt{x}\,\frac{\partial}{\partial x}$
and $a_k$ constants sa\-tisfying 
$a_1+a_2+a_3+a_4=a_1^2+a_2^2+a_3^2+a_4^2=0$;

\item[$\diamond$] rigid of $t$-degree $3$: 
$\mathrm{Ch}(X,\sigma X,\sigma^2 X,\sigma^3 X)$
with $X$ of $t$-degree $3$ and $\sigma$ a root 
of unity of order $4$;

\item[$\diamond$] special $\mathrm{Ch}(X,-X,Y,-Y)$ 
with $X$ and $Y$ of $t$-degree $2$.
\end{itemize}}
\end{conjecture}

\begin{rem}
The classification of $p$-chambars on $\mathbb{C}$ for 
$p\geq 4$ is a difficult problem in particular because 
of irreducibility problems. Indeed if $p=6$ for instance 
one can consider the vector field 
$Z_5=5x^{\frac{4}{5}}\frac{\partial}{\partial x}$ to
which one can associate the $6$-chambar
\[
\mathrm{Ch}\big(Z_5,\sigma Z_5,\sigma^2 Z_5,\sigma^3 Z_5,\sigma^4 Z_5,\sigma^5 Z_5\big)
\]
which is irreducible. But one can also consider the 
non-irreducible $6$-chambar obtained as follows
\[
\mathrm{Ch}\Big(X_1,\mathbf{j}X_1,\mathbf{j}^2X_1,X_2,\mathbf{j}X_2,\mathbf{j}^2X_2\Big)
\]
where 
$X_k=\sqrt{\lambda_kx+\mu_k}\frac{\partial}{\partial x}$
and $\lambda_k$, $\mu_k$ are complex numbers such that 
$\lambda_1\mu_2-\lambda_2\mu_1\not=0$.
\end{rem}

\begin{prob}
Classify irreducible $p$-chambars in dimension $1$, for $p\geq 4$.
\end{prob}

\begin{thm}\label{thm:4chamonC1}
{\sl Let $\mathrm{Ch}(X_1,X_2,X_3,X_4)$ be a holomorphic 
$4$-chambar on some open set $\mathcal{U}\subset\mathbb{C}$.
Set $X_k=y_k(x)\,\frac{\partial}{\partial x}$ with
$y_k\in\mathcal{O}(\mathcal{U})$ for $1\leq k\leq 4$. 

Then there exists a polynomial $P\colon\mathbb{C}^3\to\mathbb{C}^4$
independent of the $y_k$'s such that 
the vector $y=(y_1,y_2,y_3,y_4)$ satisfies a differential 
equation of the form 
\begin{equation}\label{eq:ode}
\Delta(y)\cdot y^{\prime\prime\prime}=P(y,y^\prime,y^{\prime\prime})
\end{equation}
where $\Delta(y)=\displaystyle\prod_{i<j}(y_j-y_i)$.

Furthermore, the polynomial $P$ is homogeneous of degree $7$.}
\end{thm}

\begin{proof}[{\sl Proof}]
Let us recall some basic facts. The operator $X_k$ on 
$\mathcal{O}(\mathcal{U})$ acts as $X_k(f)=y_k\cdot f'$. In particular
\begin{align*}
& X_k(x)=y_k, && X^2_k(x)=y_ky^\prime_k, && X^3_k(x)=p(y_k,y^\prime_k)+y_k^2y^{\prime\prime}_k, && X^4_k(x)=q(y_k,y_k^\prime,y_k^{\prime\prime})+y_k^3y_k^{\prime\prime\prime}
\end{align*}
where $p(y,z)=yz^2$ and $q(y,z,w)=yz^3+4y^2zw$. More generally we 
have
\begin{equation}\label{eq:rec}
X_k^\ell(x)=P_\ell\big(y_k,y^\prime_k,\ldots,y_k^{(\ell-2)}\big)+y_k^{\ell-1}\cdot y_k^{(\ell-1)}
\end{equation}
where $P_\ell$ denotes a homogeneous polynomial of  degree $\ell$.

Using (\ref{eq:rec}) we get by an induction argument
\begin{equation}\label{eq:rec2}
\frac{\partial^n X_k^\ell(x)}{\partial x^n}=P_{\ell,n}\big(y_k,y_k^\prime,\ldots,y_k^{(\ell+n-2)}\big)+y_k^{\ell-1}\cdot y_k^{(\ell+n-1)}
\end{equation}
where $P_{\ell,n}$ is homogeneous of degree $\ell$ and $P_{\ell,0}=P_\ell$.
Note that $P_{\ell,n}$ is independent of the open set $\mathcal{U}$
and of the function $y\colon\mathcal{U}\to\mathbb{C}^4$. 

Since the $X_k$'s satisfy the barycentric condition, we have
$\displaystyle\sum_{k=1}^4X_k^\ell(x)=0$, $1\leq k\leq 3$ and so 
\[
\displaystyle\sum_{k=1}^4\frac{\partial^nX_k^{\ell}(x)}{\partial x^n}=0 \quad\forall\,1\leq \ell\leq 4, \quad\forall\,n\geq 0.
\]
From the above relations and (\ref{eq:rec}) we get the following
system of equations
\[
\left\{
\begin{array}{llll}
y_1^{\prime\prime\prime}+y_2^{\prime\prime\prime}+y_3^{\prime\prime\prime}+y_4^{\prime\prime\prime}=0\\
y_1y_1^{\prime\prime\prime}+y_2y_2^{\prime\prime\prime}+y_3y_3^{\prime\prime\prime}+y_4y_4^{\prime\prime\prime}=Q_2(y,y^\prime,y^{\prime\prime})\\
y_1^2y_1^{\prime\prime\prime}+y_2^2y_2^{\prime\prime\prime}+y_3^2y_3^{\prime\prime\prime}+y_4^2y_4^{\prime\prime\prime}=Q_3(y,y^{\prime},y^{\prime\prime})\\
y_1^3y_1^{\prime\prime\prime}+y_2^3y_2^{\prime\prime\prime}+y_3^3y_3^{\prime\prime\prime}+y_4^3y_4^{\prime\prime\prime}=Q_4(y,y^{\prime},y^{\prime\prime})
\end{array}
\right.
\]
with
\[
\left\{
\begin{array}{llll}
Q_2(y,y^\prime,y^{\prime\prime})=-3\displaystyle\sum_{i=1}^4y_i^\prime y_i^{\prime\prime}\\
Q_3(y,y^\prime,y^{\prime\prime})=-\displaystyle\sum_{i=1}^4\Big((y_i^\prime)^3+4y_iy_i^\prime y_i^{\prime\prime}\Big)\\
Q_4(y,y^\prime,y^{\prime\prime})=-\displaystyle\sum_{i=1}^4\Big(y_i(y_i^\prime)^3+4y_i^2y_i^\prime y_i^{\prime\prime}\Big)\\
\end{array}
\right.
\]

Writing the above system in the matrix form we get $W(y)\cdot {}^{\mathrm{t}}\! (y^{\prime\prime\prime})={}^{\mathrm{t}} \!  Q(y,y^{\prime},y^{\prime\prime})$ where ${}^{\mathrm{t}} \! v$ denotes the transpose of $v$ and 
\[
W=\left(\begin{array}{cccc}
1 & 1 & 1 & 1 \\
y_1 & y_2 & y_3 & y_4\\
y_1^2 & y_2^2 & y_3^2 & y_4^2\\
y_1^3 & y_2^3 & y_3^3 & y_4^3\\
\end{array}\right)
\]

Solving (\ref{eq:rec2}) we get that the vector function $y$ satisfies the ODE
\begin{equation}\label{eq:rec3}
\Delta{}^{\mathrm{t}} \!(y^{\prime\prime\prime})=\mathrm{adj}(W)(y)\cdot {}^{\mathrm{t}} \! Q(y,y^\prime,y^{\prime\prime})
\end{equation}
where $\mathrm{adj}(W)$ is the adjoint of the matrix $W$, $\Delta=\det(W)=\displaystyle\prod_{i<j}(y_j-y_i)$ and $Q=(0,Q_2,Q_3,Q_4)$.
Set $P(y,y^\prime,y^{\prime\prime})=\mathrm{adj}(W)(y)\cdot {}^{\mathrm{t}} \! Q(y,y^\prime,y^{\prime\prime})$. 
By looking carefully at the right hand side of the above relation, we see that 
$P$ is homogeneous of degree $7$.
\end{proof}

\begin{rems}
Let us fix three (constant) vectors $\alpha_0$, $\alpha_1$
and $\alpha_2$ in $\mathbb{C}^4$ and assume that the 
components of $\alpha_0$ are two by two different. Then 
there exists an unique germ 
$y=(y_1,y_2,y_3,y_4)\in\mathcal{O}(\mathbb{C}^4,0)$ 
satisfying (\ref{eq:ode}) with initial conditions 
$y(0)=\alpha_0$, $y'(0)=\alpha_1$ and $y''(0)=\alpha_2$.

Since the differential equation (\ref{eq:ode}) is 
meromorphic on $\mathbb{C}^4$ the solution $x\mapsto y(x)$
can be extended until it reaches the codimension one 
submanifold $\displaystyle\bigcup_{i<j}(y_i=y_j)$ of $\mathbb{C}^4$.

For instance, the constant vectors $y=(a_1,a_2,a_3,a_4)$ 
are solutions of the ODE (\ref{eq:rec3}). In fact, if $y$ is
a constant vector then $y^\prime=y^{\prime\prime}=0$ and
$Q(y,y^\prime,y^{\prime\prime})=0$. 
\end{rems}

Next we will study the solutions with initial condition of
the form $y_i(0)=y_j(0)$, $i\not=j$. The idea is to lift
the ODE to a first order ODE on $\mathbb{C}^{12}$. 

Consider the ODE (\ref{eq:ode}) of order $3$ on $\mathcal{U}\subset\mathbb{C}^4$.
Introducing new variables $z=y^\prime$, and $w=z^\prime=y^{\prime\prime}$, 
this ODE can be lifted to a system of meromorphic ODE's of order
$1$ on $\mathcal{V}=\mathcal{U}\times\mathbb{C}^4\times\mathbb{C}^4$
as 
\begin{equation}\label{eq:sys}
\left\{
\begin{array}{lll}
y'=z\\
z'=w\\
w'=\Delta^{-1}\cdot P(y,z,w)
\end{array}
\right.
\end{equation}
Multiplying (\ref{eq:sys}) by $\Delta$ we obtain a tangent 
holomorphic vector field on $\mathcal{V}$
\begin{equation}\label{eq:vf}
\chi(y,z,w)=\Delta\displaystyle\sum_{j=1}^4z_j\frac{\partial}{\partial y_j}+\Delta\displaystyle\sum_{j=1}^4w_j\frac{\partial}{\partial z_j}+\displaystyle\sum_{j=1}^4P_j(y,z,w)\frac{\partial}{\partial w_j}.
\end{equation}

\begin{thm}\label{thm:4chamonC2}
{\sl The following submanifolds of $\mathbb{C}^{12}$ are 
$\chi$-invariant:
\begin{itemize}
\item[$\diamond$] $\Sigma_{ij}:=\mathcal{Z}\big(\langle y_j-y_i\rangle\big)$ for any $1\leq i<j\leq 4$;
\smallskip
\item[$\diamond$] $\Sigma_1:=\mathcal{Z}\big(\langle\displaystyle\sum_j y_j,\displaystyle\sum_j z_j,\displaystyle\sum_j w_j\rangle\big)$;
\smallskip
\item[$\diamond$] $\Sigma_2:=\mathcal{Z}\big(\langle\displaystyle\sum_j y_jz_j,\displaystyle\sum_j (z_j^2+y_jw_j)\rangle\big)$;
\smallskip
\item[$\diamond$] $\Sigma_3:=\mathcal{Z}\big(\langle\displaystyle\sum_j(y_jz_j^2+y_j^2w_j)\rangle\big)$.
\end{itemize}
The notation $\mathcal{Z}(\mathcal{J})$ stands for the 
zeroes of the ideal $\mathcal{J}$.}
\end{thm}

All these submanifolds are complete intersections and the codimensions 
coincide with the number of generators of the ideal. 
Furthermore, the submanifolds $\Sigma_i$, $1\leq i\leq 3$, 
coincide with the initial conditions corresponding to the 
barycentric conditions
\[
\displaystyle\sum_{k=1}^4\frac{\partial^nX_k^{\ell}}{\partial x^n}=0 \quad\forall\,1\leq n+\ell\leq 4, \quad\forall\,n\geq 0.
\]

Let us now give a Lemma that will be useful for the proof of 
Theorem \ref{thm:4chamonC2}.

\begin{lem}\label{lem:teccom}
{\sl The components $P_1$, $P_2$, $P_3$, $P_4$ of $\chi$
satisfy the following relations:
\begin{itemize}
\item[$\diamond$] $\displaystyle\sum_iP_i=0$,

\item[$\diamond$] $\displaystyle\sum_iy_iP_i=\Delta Q_2(y,z,w)=-3\Delta\displaystyle\sum_iz_iw_i$,

\item[$\diamond$] $\displaystyle\sum_iy_i^2P_i=\Delta Q_3(y,z,w)=-\Delta\displaystyle\sum_i(z_i^3+4y_iz_iw_i)$,

\item[$\diamond$] $\displaystyle\sum_iy_i^3P_i=\Delta Q_4(y,z,w)=-\Delta\displaystyle\sum_i(y_iz_i^3+4y_i^2z_iw_i)$.
\end{itemize}}
\end{lem}

\begin{proof}[{\sl Proof}]
Recall that on the one hand
\[
{}^{\mathrm{t}} \! P(y,y^\prime,y^{\prime\prime})=\mathrm{adj}(W)(y){}^{\mathrm{t}} \! Q(y,y^\prime,y^{\prime\prime})
\]
so 
\[
{}^{\mathrm{t}} \! P(y,z,w)=\mathrm{adj}(W)(y){}^{\mathrm{t}} \! Q(y,z,w).
\]
On the other hand the four relations of the statement
are equivalent to $W(y){}^{\mathrm{t}} \! P(y,z,w)=\Delta {}^{\mathrm{t}} \! Q(y,z,w)$.
Finally, we know from linear algebra that 
$W(y)\mathrm{adj}(W)(y)=\Delta\cdot \mathrm{id}$, where $\mathrm{id}$
is the identity matrix. As a consequence
\[
W(y){}^{\mathrm{t}} \! P(y,z,w)=W(y)\mathrm{adj}(W)(y){}^{\mathrm{t}} \! Q(y,z,w)=\Delta {}^{\mathrm{t}} \! Q(y,z,w).
\]
\end{proof}

\begin{proof}[{\sl Proof of Theorem \ref{thm:4chamonC2}}]
Let $\mathcal{J}$ be an ideal of $\mathbb{C}[y,z,w]$.
Recall that the submanifold $\mathcal{Z}(\mathcal{J})$, 
defined by $\mathcal{J}$, is $\chi$-invariant if, 
and only if, $\chi(\mathcal{J})\subset\mathcal{J}$. 
So, for instance
\[
\chi(y_k-y_\ell)=(z_k-z_\ell)\prod_{i<j}(y_j-y_i)
\]
and $\chi(y_k-y_\ell)$ belongs to $\langle y_k-y_\ell\rangle$; 
in particular $\Sigma_{k\ell}$ is $\chi$-invariant. 

\smallskip

Consider the ideal $\mathcal{J}_1=\langle\displaystyle\sum_j y_j,\displaystyle\sum_j z_j,\displaystyle\sum_j w_j\rangle$. 
We have
\begin{align*}
& \chi\Big(\displaystyle\sum_iy_i\Big)=\displaystyle\sum_i\chi(y_i)=\Delta\displaystyle\sum_iz_i\in\mathcal{J}_1\\
&\chi\Big(\displaystyle\sum_iz_i\Big)=\displaystyle\sum_i\chi(z_i)=\Delta\displaystyle\sum_iw_i\in\mathcal{J}_1\\
& \chi\Big(\displaystyle\sum_iw_i\Big)=\displaystyle\sum_i\chi(w_i)=\displaystyle\sum_i P_i=0\in\mathcal{J}_1 \text{ by the first assertion of Lemma \ref{lem:teccom}}
\end{align*}
With a similar computation, using the other assertions of Lemma \ref{lem:teccom}
it is possible to prove that $\Sigma_1$, $\Sigma_2$ 
and~$\Sigma_3$ are $\chi$-invariant.
\end{proof}

\begin{cor}
{\sl Let $\mathrm{Ch}(X_1,X_2,X_3,X_4)$ be a $4$-chambar
on an open set $\mathcal{U}\subset\mathbb{C}$, with 
$X_j=y_j\frac{\partial}{\partial x}$, $y_j\in\mathcal{O}(\mathcal{U})$, 
$1\leq j\leq 4$.

Suppose that $y_k(x_0)=y_{\ell}(x_0)$ and that $P(y_0,z_0,w_0)\not=0$
 for some initial condition and $k\not=\ell$.
Then $y_k(x)=y_\ell(x)$ for all $x\in\mathcal{U}$.
Moreover, if $k=1$ and $\ell=2$, for instance,
then either the chambar is constant and $2a_1+a_3+a_4=0$ or 
$y_j(x)=a_j\sqrt{\lambda x+\mu}$ with $\lambda\not=0$, 
$a_1=a_2=-\frac{1}{3}$ and $a_3$ and $a_4$ the roots of $3z^2+2z+3=0$.}
\end{cor}

\begin{proof}[{\sl Proof}]
According to Theorem \ref{thm:4chamonC1} if $y_1\frac{\partial}{\partial x_1}$, $y_2\frac{\partial}{\partial x_2}$,
$\ldots$, $y_4\frac{\partial}{\partial x_4}$ are
holomorphic vector fields that define a $4$-chambar on an open 
set $\mathcal{U}\subset\mathbb{C}$, 
then the vector function $x\in\mathcal{U}\mapsto y(x)=\big(y_1(x),y_2(x),\ldots,y_4(x)\big)$ satisfies an ODE 
of the form 
\begin{equation}\label{eq:odeb}
\Delta\,y^{\prime\prime\prime}=P(y,y^\prime,y^{\prime\prime}),
\end{equation}
where $\Delta=\displaystyle\prod_{i<j}(y_j-y_i)$. 

Assume that $y_1(x_0)=y_2(x_0)$. Since 
$P(y(x_0),z(x_0),w(x_0))\not=0$
the point $(y(x_0),z(x_0),w(x_0))$ is not a singularity of the vector 
field $\chi$, and there is only one solution through this point.
Using that the set  
$\{y_1=y_2\}$ is $\chi$-invariant, we get $y_1(x)=y_2(x)$ for any 
$x\in\mathcal{U}$.

The condition on the flows is now
\[
2\varphi_t^1(x)+\varphi_t^3(x)+\varphi_t^4(x)=4x,
\]
which is a particular case of $(\ref{eq:chambweigh})$.
\end{proof}

A natural question is the following:

\begin{que}
What could happen in the case 
$P\big(y(x_0),z(x_0),w(x_0)\big)~=~0$ and $y_k(x_0)=y_\ell(x_0)$ ? 
Are there solutions 
with these conditions and $y_k(x)\not\equiv y_\ell(x)$, but
$(y(x),z(x),w(x))\in\Sigma=\Sigma_1\cap\Sigma_2\cap\Sigma_3$ for all $x\in \mathcal{U}$ ?
\end{que}

Let us denote by $\mathrm{Ch}(4,1)$ the set of $4$-tuples 
$(X_1,X_2,X_3,X_4)$ of germs at $0\in~\mathbb{C}$ of
holomorphic vector fields whose flows satisfy the 
barycentric conditions.

\begin{cor}
{\sl The set $\mathrm{Ch}(4,1)$ is isomorphic to an algebraic
submanifold of~$\mathbb{C}^{12}$ whose irreducible 
components have dimension at most six.}
\end{cor}

\begin{proof}[{\sl Proof}]
According to Theorems \ref{thm:4chamonC1} and 
\ref{thm:4chamonC2} 
any $4$-chambar on $\mathbb{C}$ gives origin to a 
trajectory $(y,z,w)\colon(\mathbb{C},0)\to\mathbb{C}^{12}$
tangent to the $\chi$-invariant submanifold 
$\Sigma=\Sigma_1\cap\Sigma_2\cap\Sigma_3$ of 
$\mathbb{C}^{12}$. The initial condition 
$(y(0),z(0),w(0))$ caracterizes the trajectory
$(y,z,w)$ and defines an embedding of 
$\mathrm{Ch}(4,1)$ on $\Sigma$.
\end{proof}

\section{Linear chambars}\label{sec:linnil}

\begin{thm}\label{thm:linnil}
{\sl Let $X_1$, $X_2$, $\ldots$, $X_p$ be some 
linear vector fields on 
$\mathbb{R}^n$ 
$($resp. $\mathbb{C}^n)$. 

If they satisfy the barycentric property, then 
they are nilpotent.}
\end{thm}

\begin{proof}[{\sl Proof}]
The flow $\varphi_t^k$ of $X_k$ can be written
\[
\varphi_t^k(x)=(\exp tA_k)(x)
\]
where the $A_k$ belong to $\mathrm{End}(\mathbb{R}^n)$
or $\mathrm{End}(\mathbb{C}^n)$. We identify 
the $A_k$ to some matrices. The barycentric property
is equivalent to 
\[
\displaystyle\sum_{k=1}^p\sum_{\ell=0}^\infty\frac{t^\ell}{\ell !}A_k^\ell=p\mathrm{Id}
\]
that implies $\displaystyle\sum_{k=1}^p A_k^\ell=0$ for 
any $\ell\geq 1$. Let $\lambda_{k,j}$ be the eigenvalues
of $A_k$, $1\leq j\leq n$. We get for all $\ell\geq 1$
\[
0=\mathrm{Tr}\Big(\displaystyle\sum_{k=1}^pA_k^\ell\Big)=\displaystyle\sum_{k=1}^p\sum_{j=1}^n\lambda_{k,j}^n.
\]
As a result all the $\lambda_{k,j}$ are equal to zero. 
\end{proof}

\begin{rem}
The $\varphi_t^k$ are polynomial in $x$ and $t$.
\end{rem}

\begin{rem}
If $p=2$, then the indices of nilpotence are $2$ 
({\it i.e.} $A^2=0$) and we recover the fact that
the trajectories are straight lines. Note also 
that if $X$ is a nilpotent vector field of 
index $2$, then the pair $(X,-X)$ is a $2$-chambar.
\end{rem}

\begin{eg}
Let $X$ be a nilpotent linear vector field of 
order $p$. Let $\sigma=\exp\left(\frac{2\mathbf{i}\pi}{p}\right)$
be a primitive $p$-th root of unity. Then 
the vector fields $X$, $\sigma X$, $\sigma^2X$, 
$\ldots$, $\sigma^{p-1}X$ satisfy the barycentric 
property.
\end{eg}

\begin{rem}
Let $\mathrm{Ch}(X_1,X_2,\ldots,X_p)$ be a linear
$p$-chambar. Denote by $k$ the maximal order of 
nilpotence of the $X_i$'s. Take $\ell<k$ an 
integer. Then 
$\mathrm{Ch}(X_1^\ell,X_2^\ell,\ldots,X_p^\ell)$
is a $q$-chambar for some $q\leq p$. The 
inequality comes from the fact that two 
$X_k^\ell$ can be equal or $X_k^\ell$ can be zero. 
The fact that $q<p$ measures some degeneration and
if $q=p$ for any $\ell<k$ it gives some
condition of transversality.
\end{rem}

\begin{rem}
Let $\mathrm{Ch}(X_1,X_2,\ldots,X_p)$ be a singular
$p$-chambar such that $X_k(0)=0$. Denote by 
$A_i$ the linear part of $X_i$ for $1\leq i\leq p$.

Assume that the $A_i$'s generate a linear $p$-chambar
$\mathrm{Ch}(A_1,A_2,\ldots,A_p)$.

Consider the homothety $h_s\colon x\mapsto sx$, 
$s\in\mathbb{C}^*$ and 
\[
X_k^s=h_{s*}X_k=A_k+s(\ldots)
\]
We construct in this way a family  
$\mathrm{Ch}^s=\mathrm{Ch}(X_1^s,X_2^s,\ldots,X_p^s)$
of $p$-chambars, all conjugate for $s\not=0$, 
and that joins the initial chambar 
$\mathrm{Ch}^1=\mathrm{Ch}(X_1,X_2,\ldots,X_p)$
to the linear chambar 
$\mathrm{Ch}^0=\mathrm{Ch}(A_1,A_2,\ldots,A_p)$.
\end{rem}

\subsection{Linear $p$-chambars in dimension $2$}

\begin{lem}
{\sl Let $B$ be $(2\times 2)$-matrix with complex coefficients.
If $\mathrm{Tr}(B)=0$, then $B$ is the sum of two 
nilpotent matrices.}
\end{lem}

\begin{proof}[{\sl First proof}]
If $B=0$, then the result holds. 

Let us now assume that $B\not=0$. Let us write 
$B$ as $\left(
\begin{array}{cc}
a & b\\
c & -a
\end{array}
\right)$. We are looking forward two nilpotent matrices
\begin{align*}
& A=\left(
\begin{array}{cc}
x & y \\
z & -x
\end{array}
\right)
&& 
A'=\left(
\begin{array}{cc}
x' & y' \\
z' & -x'
\end{array}
\right)
\end{align*}
such that $B=A+A'$. We thus have to solve
the following system
\[
\left\{
\begin{array}{lllll}
x+x'=a\\
y+y'=b\\
z+z'=c\\
x^2+yz=0\\
x^{'2}+y'z'=0
\end{array}
\right.
\]
(the last two conditions guaranteeing nilpotence). 
After elimination of $x'$, $y'$ and $z'$ we get
\[
\left\{
\begin{array}{ll}
x^2+yz=0\\
(a-x)^2+(b-y)(c-z)=0
\end{array}
\right.
\]
that is
\[
\left\{
\begin{array}{ll}
x^2+yz=0\\
2ax+bz+cy-a^2-bc=0
\end{array}
\right.
\]
which is the non-trivial intersection of a 
quadric and of a plane. These two sets 
intersect along a plane conic.
\end{proof}

\begin{proof}[{\sl Second proof}]
Since $\mathrm{Tr}(B)=0$, then $B$ is conjugate to 
$\left(\begin{array}{cc}0 & x\\y & 0\end{array}\right)$
for some $x$, $y$ in $\mathbb{C}$ (note 
that if $B$ is nilpotent, then $y=0$).
We conclude using the fact that 
\[
\left(\begin{array}{cc}0 & x\\y & 0\end{array}\right)=\underbrace{\left(\begin{array}{cc}0 & x\\0 & 0\end{array}\right)}_{\text{nilpotent}}+\underbrace{\left(\begin{array}{cc}0 & 0\\y & 0\end{array}\right)}_{\text{nilpotent}}
\]
\end{proof}

\begin{cor}
{\sl Let $A_3$, $A_4$, $\ldots$, $A_p$ be $(p-2)$ nilpotent
$(2\times 2)$-matrices.

There exist two nilpotent $(2\times 2)$-matrices $A_1$, $A_2$ such 
that the flows $\varphi_t^k=\exp tA_k$, $1\leq k\leq p$, 
satisfy the barycentric property.}
\end{cor}

\begin{proof}[{\sl Proof}]
Let $A_1$ and $A_2$ be two nilpotent matrices such that
\[
A_1+A_2+A_3+\ldots+A_p=0.
\]
As $\exp tA_k=\mathrm{Id}+tA_k$ in dimension $2$, the 
$p$-tuple $(A_1,A_2,\ldots,A_p)$ suits.
\end{proof}

\begin{rem}
If $A_1$, $A_2$, $A_3$ are nilpotent $(2\times 2)$-matrices
that satisfy the barycentric pro\-perty, then the $A_i$ 
are $\mathbb{C}$-colinear, {\it i.e.} $\mathrm{Ch}(A_1,A_2,A_3)$ is rigid. Indeed the nilpotent 
$(2\times 2)$-matrices form a quadratic cone.
\end{rem}

\subsection{Linear $3$-chambars}

The following example illustrates that we can find 
solutions to the barycentric property in some 
Lie algebras of vector fields. In the particular
case $n=3$ one can find $p$-chambars in the 
Heisenberg Lie algebra~$\mathfrak{h}_3$ formed by matrices 
\[
M(\alpha,\beta,\gamma)=\left(
\begin{array}{ccc}
0 & \alpha & \gamma \\
0 & 0 & \beta \\
0 & 0 & 0
\end{array}
\right).
\]
One has $M^2(\alpha,\beta,\gamma)=M(0,0,\alpha\beta)$.
The barycentric property for the vector fields $X_k$ 
corresponding to the matrices $M(\alpha_k,\beta_k,\gamma_k)$, 
$k=1$, $\ldots$, $p$, is equivalent to the equalities
\[
\displaystyle\sum_{k=1}^p\alpha_k=\sum_{k=1}^p\beta_k=\sum_{k=1}^p\gamma_k=\sum_{k=1}^p\alpha_k\beta_k=0.
\]
In the coefficients space $(\mathbb{C}^3)^p$ 
the barycentric property is the intersection 
of three hyperplanes and one quadric
which has thus dimension $3p-4$.

\begin{thm}\label{thm:lin3cham}
{\sl Let $\mathrm{Ch}(X_1,X_2,X_3)$ be a linear $3$-chambar
on $\mathbb{C}^3$. Then up to conjugacy, the~$X_i$'s 
$($identified to their matrices$)$ are contained in the 
Heisenberg Lie algebra $\mathfrak{h}_3~\subset~\mathrm{gl}(3,\mathbb{C})$.}
\end{thm}

\begin{proof}[{\sl Proof}]
Let us identified $X_i$ to its matrix. 

We will 
distinguish two cases according to the rank of 
the $X_i$'s.
\begin{itemize}
\item[$\diamond$] If one of the $X_i$'s has rank $2$, 
for instance $X_1$, then 
up to conjugacy one can assume that 
$X_1=\left(
\begin{array}{ccc}
0 & 1 & 0\\
0 & 0 & 1\\
0 & 0 & 0
\end{array}
\right)$. We are now looking for $X_2$ and~$X_3$ 
such that $X_2$ and $X_3$ are nilpotent 
(in particular their traces are zero) and
$X_1+X_2+X_3=X_1^2+X_2^2+X_3^2=0$. A 
straightforward computation implies that 
$X_2$ and $X_3$ belong to $\mathfrak{h}_3$.

\item[$\diamond$] It suffices now to deal with 
the case where the three nilpotent matrices $X_1$, 
$X_2$ and $X_3$ have rank $1$. Up to conjugacy
one can suppose that 
$X_1=\left(
\begin{array}{ccc}
0 & 0 & 1\\
0 & 0 & 0\\
0 & 0 & 0
\end{array}
\right)$. As $X_2$ has rank $1$ the three 
columns of $X_2$ are colinear, {\it i.e.}
$X_2=(\lambda E,\mu E,\nu E)$ where 
$E=\left(
\begin{array}{c}
a\\
b\\
c
\end{array}
\right)\not=0$. 
Then $X_3=-X_1-X_2=\left(-\lambda E,-\mu E,-\nu E-\left(
\begin{array}{c}
1\\
0\\
0
\end{array}\right)\right)$. Let us distinguish
three cases:
\begin{itemize}
\item[$\bullet$] First assume that 
$\lambda=\mu=0$. Changing the notations if 
needed let us take $\nu=1$. Then 
\begin{align*}
&X_1=\left(
\begin{array}{ccc}
0 & 0 & 1\\
0 & 0 & 0\\
0 & 0 & 0
\end{array}
\right), &&
X_2=\left(
\begin{array}{ccc}
0 & 0 & a\\
0 & 0 & b\\
0 & 0 & c
\end{array}
\right), &&
X_3=-\left(
\begin{array}{ccc}
0 & 0 & a+1\\
0 & 0 & b\\
0 & 0 & c
\end{array}
\right).
\end{align*}
Since $X_1$ and $X_2$ are nilpotent,
$c$ has to be $0$; but $c=0$ leads to 
$X_2^2=X_3^2=0$, and the $X_i$ belong
to $\mathfrak{h}_3$.

\item[$\bullet$] Now suppose $\lambda\not=0$, 
{\it i.e.} $\lambda=1$. Then 
\begin{align*}
& X_2=\left(
\begin{array}{ccc}
a & \mu a & \nu a\\
b & \mu b & \nu b\\
c & \mu c & \nu c
\end{array}
\right), &&X_3=-\left(
\begin{array}{ccc}
a & \mu a & \nu a+1\\
b & \mu b & \nu b\\
c & \mu c & \nu c
\end{array}
\right).
\end{align*}
As $X_3$ has rank $1$, the coefficients 
$b$ and $c$ are zero. Therefore 
$X_2=\left(
\begin{array}{ccc}
a & \mu a & \nu a\\
0 & 0 & 0\\
0 & 0 & 0
\end{array}
\right)$; since $X_2$ is nilpotent, 
$a$ has to be $0$. As a consequence
$X_2=0$ which is impossible (the matrices are 
implicitly 
assumed to be non-zero).

\item[$\bullet$] Finally assume that $\lambda=0$
and $\mu\not=0$, that is $\lambda=0$ and 
$\mu=1$ and 
\begin{align*}
& X_2=\left(
\begin{array}{ccc}
0 & a & \nu a\\
0 & b & \nu b\\
0 & c & \nu c
\end{array}
\right), &&X_3=-\left(
\begin{array}{ccc}
0 & a & \nu a+1\\
0 & b & \nu b\\
0 & c & \nu c
\end{array}
\right).
\end{align*}
The fact that $\mathrm{rk}\,X_3=1$ leads to 
$b=c=0$ and 
\begin{align*}
& X_2=\left(
\begin{array}{ccc}
0 & a & \nu a\\
0 & 0 & 0\\
0 & 0 & 0
\end{array}
\right), &&X_3=-\left(
\begin{array}{ccc}
0 & a & \nu a+1\\
0 & 0 & 0\\
0 & 0 & 0
\end{array}
\right)
\end{align*}
belong to $\mathfrak{h}_3$.
\end{itemize}
\end{itemize}
\end{proof}

In fact the statement holds in any dimension but we keep 
the previous result and its proof because 
this last one is much more easier. Let us start
by some definitions, notations and intermediate
results of non-commutative algebra.

A \textsl{monomial} of $k$-variables on 
$\mathrm{End}(\mathbb{C}^n)$ is a map 
$f\colon\mathrm{End}(\mathbb{C}^n)^k\to\mathrm{End}(\mathbb{C}^n)$
of the form 
\[
f(X_1,X_2,\ldots,X_k)=X_{i_1}^{k_1}X_{i_2}^{k_2}\ldots X_{i_r}^{k_r}
\]
where $r\geq 1$, $i_j\in\{1,\,2,\,\ldots,\,k\}$ and 
$k_j\geq 0$ for any $1\leq j\leq r$. By convention 
$X_i^0=1$.

We say that the monomial is \textsl{reduced} if 
\begin{itemize}
\item[$\diamond$] $k_j\geq 1$ for any $1\leq j\leq r$;

\item[$\diamond$] $i_j\not= i_{j+1}$ for any $1\leq j\leq r-1$.
\end{itemize}

The \textsl{degree} of $f$ is $\deg f=\displaystyle\sum_{i=1}^rk_i$.
A \textsl{polynomial of $k$ variables} on $\mathrm{End}(\mathbb{C}^n)$
is a linear combination of monomials of $k$ variables on 
$\mathrm{End}(\mathbb{C}^n)$:
\[
P(X_1,X_2,\ldots,X_k)=\displaystyle\sum_{j=1}^sa_jF_j(X_1,X_2,\ldots,X_k)
\]
with $a_1$, $a_2$, $\ldots$, $a_s$ in $\mathbb{C}$. 
The \textsl{degree} of $P$ is 
$\deg P=\max\{\deg(F_j)\,\vert\, a_j\not=0\}$.
If $\deg F_j\geq 1$ for any $1\leq j\leq s$, then we say that 
$P$ is \textsl{without constant term}.

If $\mathrm{Ch}(X_1,X_2,X_3)$ is a $3$-linear chambar on $\mathbb{C}^n$,
we denote by 
$\mathcal{G}=\langle X_1,\,X_2,\,X_3\rangle\subset\mathrm{End}(\mathbb{C}^n)\simeq\mathrm{gl}(n,\mathbb{C})$
the sub-algebra generated by $X_1$, $X_2$ and $X_3$. As previously we identify the
linear vector field $X_j$ with 
elements of $\mathrm{End}(\mathbb{C}^n)$.

We can now state the result:

\begin{thm}\label{thm:3linchambar}
{\sl Let $\mathrm{Ch}(X_1,X_2,X_3)$ be a linear $3$-chambar
on $\mathbb{C}^n$. Let $\mathcal{G}=\langle X_1,\,X_2,\,X_3\rangle$
be the algebra of linear transformations generated by $X_1$, 
$X_2$ and~$X_3$.

If $Y_1$, $Y_2$, $\ldots$, $Y_n$ belong to $\mathcal{G}$, 
then $Y_1Y_2\ldots Y_n=0$.

In particular, up to conjugacy, the $X_i$'s 
$($identified to their matrices$)$ are contained in the 
Heisenberg Lie algebra $\mathfrak{h}_n\subset\mathrm{gl}(n,\mathbb{C})$.}
\end{thm}

\begin{proof}[{\sl Proof}]
Let us start the proof with the following statement:
\begin{lem}\label{lem:tec1}
{\sl Let $\mathrm{Ch}(X_1,X_2,X_3)$ be a linear $3$-chambar
on $\mathbb{C}^n$. 

Let $f$ be a monomial of two variables on 
$\mathrm{End}(\mathbb{C}^n)$. 

There exists $n(f)\in\mathbb{Z}$ such that
\[
f(X_1,X_2)+f(X_2,X_1)=n(f)\cdot X_3^{\deg f}.
\]}
\end{lem}

\begin{proof}[{\sl Proof}]
For instance from
\[
X_1^k+X_2^k=-X_3^k\qquad\forall\,k\geq 1
\]
we get 
\[
X_3^{k+j}=(X_1^k+X_2^k)(X_1^j+X_2^j)=X_1^{k+j}+X_2^{k+j}+X_1^kX_2^j+X_2^kX_1^j=-X_3^{k+j}+X_1^kX_2^j+X_2^kX_1^j
\]
and so $X_1^kX_2^j+X_2^kX_1^j=2X_3^{k+j}$.

A reduced monomial $g$ of two variables on 
$\mathrm{End}(\mathbb{C}^n)$ can be written 
as 
\[
g(X,Y)=X^{k_1}Y^{j_1}X^{k_2}\ldots Y^{j_r}
\]
where $k_1\geq 0$, $j_r\geq 0$, $k_2$, $k_3$, 
$\ldots$, $k_r\geq 1$ and $j_1$, $j_2$, 
$\ldots$, $j_{r-1}\geq 1$. Note that 
$\deg g=\displaystyle\sum_{i=1}^r(k_i+j_i)$. Let us 
introduce the following definitions:
\begin{itemize}
\item[$\diamond$] the $X$-length of $g$ is $\ell_X(g)=\#\{i\,\vert\,k_i>0\}$;

\item[$\diamond$] the $Y$-length of $g$ is $\ell_Y(g)=\#\{i\,\vert\,j_i>0\}$;

\item[$\diamond$] the length of $g$ is $\ell(g)=\ell_X(g)+\ell_Y(g)$.
\end{itemize}

The proof is by induction on $\ell(f)$. Let us 
state the induction assumption: given $m\in\mathbb{N}$
the assertion of the Lemma is true for any reduced
monomial $g$ with $\ell(g)\leq m$. 

The induction assumption is true if $m\leq 2$: 
\begin{itemize}
\item[$\diamond$] for 
$\ell(f)=1$ it is a consequence of the 
equality $X_1^k+X_2^k=-X_3^k$;

\item[$\diamond$] for $\ell(f)=2$ it is 
a consequence of the equality 
$X_1^kX_2^j+X_2^kX_1^j=2X_3^{k+j}$.
\end{itemize}

Assume that the assertion of the Lemma is true for 
$m\geq 2$ and let us prove that it is true for 
$m+1$. Let $f$ be a monomial with length 
$m+1\geq 3$. Without loss of generality 
we can assume that $f(X,Y)=X^kY^jX^mg(X,Y)$; 
note that $\ell(g)=\ell(f)-3=m-2$. Using that
$X_1^kX_2^j+X_2^kX_1^j=2X_3^{k+j}$ we have
\begin{eqnarray*}
& & f(X_1,X_2)+f(X_2,X_1) \\
& & \hspace{1cm} =X_1^kX_2^jX_1^mg(X_1,X_2)+X_2^kX_1^jX_2^mg(X_2,X_1)\\
& & \hspace{1cm} = (2X_3^{k+j}-X_2^kX_1^j)X_1^mg(X_1,X_2)+(2X_3^{k+j}-X_1^kX_2^j)X_2^mg(X_2,X_1)\\
& & \hspace{1cm} = 2X_3^{k+j}(X_1^mg(X_1,X_2)+X_2^mg(X_2,X_1))-X_2^kX_1^{j+m}g(X_1,X_2)-X_1^kX_2^{j+m}g(X_2,X_1)\\
& & \hspace{1cm} = 2X_3^{k+j}\Big(g_1(X_1,X_2)+g_1(X_2,X_1)\Big)-\Big(g_2(X_1,X_2)+g_2(X_2,X_1)\Big)
\end{eqnarray*}
where $g_1(X,Y)=X^mg(X,Y)$ and $g_2(X,Y)=Y^kY^{j+m}g(X,Y)$.
Note that 
\begin{align*}
& \ell(g_1)=1+\ell(g)=m-1 && \text{and} &&\ell(g_2)=\ell(g)+2=m.
\end{align*}
Therefore the induction assumption implies that for $i\in\{1,\,2\}$
\[
g_i(X_1,X_2)+g_i(X_2,X_1)=\ell(g_i)X_3^{\deg f}.
\]
Hence 
\[
f(X_1,X_2)+f(X_2,X_1)=\ell(f)X_3^{\deg f}
\]
where $\ell(f)=2\ell(g_1)-\ell(g_2)$.
\end{proof}

\begin{lem}\label{lem:tec2}
{\sl Let $\mathrm{Ch}(X_1,X_2,X_3)$ be a linear $3$-chambar
on $\mathbb{C}^n$. 

Let $P(X,Y)$ be a polynomial of two variables on 
$\mathrm{End}(\mathbb{C}^n)$. 
Assume that $P$ is without constant term. 

Then $P(X_1,X_2)$ is 
nilpotent, that is $P(X_1,X_2)^n=0$.}
\end{lem}

\begin{proof}[{\sl Proof}]
Assume first that $P$ is a reduced monomial. Set $d=\deg P$.
Denote by $\lambda_1$, $\lambda_2$, $\ldots$, $\lambda_n$
(resp. by $\mu_1$, $\mu_2$, $\ldots$, $\mu_n$) the eigenvalues
of $P(X_1,X_2)$ (resp. $P(X_2,X_1)$). It follows from 
Lemma \ref{lem:tec1} that 
\[
\displaystyle\sum_j\lambda_j+\sum_j\mu_j=\mathrm{tr}\big(P(X_1,X_2)+P(X_2,X_1)\big)=\mathrm{tr}\big(n(P)X_3^d\big)=0.
\]
Given any $m\in\mathbb{N}$, since $P(X,Y)^m$
is a monomial we have
\[
\sum_j\lambda_j^m+\sum_j\mu_j^m=\mathrm{tr}\big(P(X_1,X_2)^m+P(X_2,X_1)^m\big)=0\qquad\forall\,m\in\mathbb{N}.
\]
This implies that 
$\lambda_1=\lambda_2=\ldots=\lambda_n=0$ and so 
$P(X_1,X_2)$ is nilpotent. In particular we get
$\mathrm{tr}(P(X_1,X_2))=0$. 

\bigskip

Suppose now that $P$ is a polynomial of two variables on 
$\mathrm{End}(\mathbb{C}^n)$ without constant term. Since $P$ is a linear
combination of non constant monomials we get 
$\mathrm{tr}(P(X_1,X_2))=0$. Similarly, given $m\in\mathbb{N}$
then $P(X,Y)^m$ is also a polynomial without constant term and so 
$\mathrm{tr}(P(X_1,X_2)^m)=~0$. Therefore 
$P(X_1,X_2)$ is nilpotent and as $P(X_1,X_2)$ belongs
to $\mathrm{End}(\mathbb{C}^n)$ we get 
$P(X_1,X_2)^n=~0$. 
\end{proof}

Let $\mathfrak{g}$ be any Lie algebra.
Recall some classical well known facts. 
If $x$ belongs to $\mathfrak{g}$, 
$y\mapsto[x,y]$ is an endomorphism 
of~$\mathfrak{g}$, which we denote $\mathrm{ad}\,x$. We say 
that $x$ is \textsl{ad-nilpotent} if 
$\mathrm{ad}\,x$ is a nilpotent endomorphism. 
If~$\mathfrak{g}$ is nilpotent, then all elements of 
$\mathfrak{g}$ are ad-nilpotent. The converse is also true, 
it is the Engel Theorem (\cite{Humphreys}). In 
particular a Lie algebra $\mathfrak{g}$ whose all elements 
are ad-nilpotent is triangularizable. 
If now $\mathfrak{g}$ 
is a matrices algebra whose all elements are nilpotent
(for the multiplication), then the algebra is 
up to conjugacy contained in the Heisenberg Lie 
algebra~$\mathfrak{h}_n$. This ends the proof of the theorem.

\end{proof}

\subsection{Some remarks on linear $4$-chambars}

As previously we will identify the vector 
field $X_i$ to its matrix. 

\begin{defi}
A $p$-chambar 
$\mathrm{Ch}(X_1,X_2,\ldots,X_p)$ has 
\textsl{rank $r$} if $r$ 
is the maximal rank of the $X_i$.
\end{defi}

Let us start with the following property:

\begin{pro}
{\sl Let $\mathrm{Ch}(X_1,X_2,X_3,X_4)$ be a linear
$4$-chambar. If $\mathrm{Ch}(X_1,X_2,X_3,X_4)$ 
has rank $2$, then 
$\mathrm{Ch}(X_1,X_2,X_3,X_4)$ is irreducible.}
\end{pro}

\begin{proof}[{\sl Proof}]
Suppose, by contradiction, that 
$\mathrm{Ch}(X_1,X_2,X_3,X_4)$ 
is reducible. Then $\mathrm{Ch}(X_1,X_2,X_3,X_4)$ 
consists of two pairs of $2$-chambars:  
the trajectories are thus lines and 
the $X_i$'s (identified with their 
matrices) have rank~$1$.
\end{proof}

\subsubsection{A first family of examples}\label{subsubsec:1ex}
Consider the four following matrices 
\begin{align*}
& X_1=\left(
\begin{array}{ccc}
0 & 0 & \alpha\\
0 & 0 & \beta\\
0 & 0 & 0
\end{array}
\right), &&
X_2=\left(
\begin{array}{ccc}
0 & \gamma & 0\\
0 & 0 & 0 \\
0 & \delta & 0  
\end{array}
\right),\\
&X_3=\left(
\begin{array}{ccc}
0& a & -\frac{ab}{c}\\
0& b & -\frac{b^2}{c}\\
0& c & -b
\end{array}
\right),&&
X_4=\left(
\begin{array}{ccc}
0& d & \frac{db}{e} \\
0& -b & -\frac{b^2}{e}\\
0& e & b
\end{array}
\right)
\end{align*}
where $ \alpha$, $\beta$, $\gamma$, 
$\delta$, $a$, $b$, $c$, 
$d$, $e$ are complex
numbers satisfying the following 
conditions
\begin{align*}
& \gamma+a+d=0, &&\alpha-\frac{ab}{c}+\frac{db}{e}=0, && \beta-\frac{b^2}{c}-\frac{b^2}{e^2}=0, && \delta+c+e=0.
\end{align*}
These matrices define a $4$-chambar 
generically irreducible whose 
elements are not contained in a nilpotent
algebra. Indeed 
\begin{itemize}
\item[$\diamond$] on the one hand the 
nilpotent algebras of matrices are 
triangularizable; in particular the 
eigenvalues of a commutator are zero;

\item[$\diamond$] on the other hand
the eigenvalues of the commutator
$[X_1,X_2]=\left(
\begin{array}{ccc}
0& \alpha\delta & -\beta\gamma \\
0& \beta\delta & 0\\
0& 0 & -\beta\delta 
\end{array}
\right)$ are non-zero as soon as
$\beta\delta\not=0$.
\end{itemize}

Remark that the $X_i$'s have a common 
kernel for generic values of the 
parameters.

\subsubsection{A second family of examples}
Let us consider 
\begin{align*}
& X_1=\left(
\begin{array}{ccc}
0 & 0 & 0\\
1 & 0 & 0\\
0 & 1 & 0
\end{array}
\right), &&
X_2=\left(
\begin{array}{ccc}
0 & a & 0\\
0 & 0 & 0 \\
b & -c-2 & 0  
\end{array}
\right),\\
&X_3=\left(
\begin{array}{ccc}
0 & -a & 0\\
0 & 0 & 0\\
b & c & 0
\end{array}
\right),&&
X_4=\left(
\begin{array}{ccc}
0 & 0 & 0 \\
-1 & 0 & 0\\
-2b & 1 & 0
\end{array}
\right)
\end{align*}

Then $(X_1,X_2,X_3,X_4)$ is a linear
$4$-chambar of rank $2$ in $\mathbb{C}^3$
and the $X_i$'s (identified to their matrices)
are not contained in a nilpotent algebra of matrices.

\bigskip

More generally for $1\leq j\leq 4$ set 
\[
X_j=\left(
\begin{array}{cc}
A_j & 0\\
B_j & 0
\end{array}
\right)
\]
where $A_j$ is a $(2\times 2)$-matrix and 
$B_j$ is a $(1\times 2)$-matrix such that
\[
\left\{
\begin{array}{ll}
A_j^2=0\\
\sum_{j=1}^4B_jA_j=0
\end{array}
\right.
\]

Then $(X_1,X_2,X_3,X_4)$ is a linear
$4$-chambar of rank $2$ in $\mathbb{C}^3$
and the $X_i$'s (identified to their matrices)
are not contained in a nilpotent algebra of matrices.

\subsubsection{A third family of examples}\label{subsubsec:3ex}

Consider 
\begin{align*}
& X_1=\left(
\begin{array}{ccc}
0 & 0 & 0\\
a & 0 & b\\
c & 0 & 0
\end{array}
\right), && X_2=\left(
\begin{array}{ccc}
0 & \alpha & 0\\
0 & 0 & 0 \\
-c & \gamma & 0
\end{array}
\right), \\
& X_3=\left(
\begin{array}{ccc}
0 & 0 & 0\\
-a & 0 & -b \\
c & 0 & 0
\end{array}
\right), && X_4=\left(
\begin{array}{ccc}
0 & -\alpha & 0\\
0 & 0 & 0 \\
-c & -\beta & 0
\end{array}
\right) 
\end{align*}
where $a$, $b$, $c$, $\alpha$, $\beta$
denote some complex numbers. Note that 
\begin{align*}
& X_1+X_2=\left(\begin{array}{ccc}
0 & \alpha & 0\\
a & 0 & b\\
0 & \beta & 0
\end{array}
\right)&& X_1+tX_2=\left(\begin{array}{ccc}
0 & t\alpha & 0\\
a & 0 & b\\
(1-t)c & t\beta & 0
\end{array}
\right) 
\end{align*}
so that 
\begin{itemize}
\item[$\diamond$] $X_1+X_2$ has rank $2$ generically 
on $a$, $b$, $\alpha$ and $\beta$, 

\item[$\diamond$] $X_1+tX_2$ has rank $3$ generically 
on $t$.
\end{itemize}

The eigenvalues of the commutator 
$[X_1,X_2]=\left(\begin{array}{ccc}
-a\alpha & 0 & -b\alpha \\
-bc & a\alpha+b\beta & 0\\
-a\beta & \alpha c & -b\beta
\end{array}
\right)$ 
are non-zero as soon as $\alpha bc\not=0$.
As a consequence $\mathrm{Ch}(X_1,X_2,X_3,X_4)$
is a $4$-chambar generically irreducible 
and the matrices associated to the $X_i$'s
are not contained in a nilportent algebra
of matrices.

Note that for generic values of parameters
the $X_i$'s do not all have the same 
kernel. As a consequence examples of 
\S\ref{subsubsec:1ex} and \S\ref{subsubsec:3ex} are not conjugated.
\bigskip

Finally one can state:

\begin{pro}
{\sl There exist linear, irreducible $4$-chambars
on $\mathbb{C}^3$
with the two following pro\-perties
\begin{itemize}
\item[$\diamond$] their flows are generically
quadratic in $t$;

\item[$\diamond$] the associated matrices are not
contained in a nilpotent algebra of matrices.
\end{itemize}}
\end{pro}

\section{Homogeneous chambars}\label{sec:homogcham}

\subsection{First properties}

Let $\mathcal{B}=\mathrm{Ch}(X_1,X_2,\ldots,X_p)$ be a 
$p$-chambar at $0\in\mathbb{C}^n$. 
We say that~$\mathcal{B}$ is 
\textsl{homogeneous of degree $\nu$} if any 
$X_i$ is homogeneous of degree~$\nu$. 

\begin{rem}
Let $\mathrm{Ch}(X,-X)$ be a homogeneous $2$-chambar on $\mathbb{C}^2$. 
Then up to linear conjugacy $X=x^\nu\frac{\partial}{\partial y}$ 
(the proof is an exercise).
\end{rem}

Given two holomorphic vector fields $X$ and $Y$ on $\mathbb{C}^n$, we define the set of colinearity between $X$ and $Y$ as
\[
\mathrm{Col}(X,Y):=\big\{m\in\mathbb{C}^n\,\vert\, X(m)\wedge Y(m)=0\big\}.
\]

\begin{rems}
We would like to remark the following facts:
\begin{itemize}
\item[$\diamond$] $\mathrm{Col}(X,Y)$ is an analytic set;

\item[$\diamond$] if $\mathrm{Col}(X,Y)\not=\emptyset$, then 
$\dim_\mathbb{C}(\mathrm{Col}(X,Y))\geq 1$;

\item[$\diamond$] if $X$ and $Y$ are homogeneous vector fields, then 
$\dim_\mathbb{C}(\mathrm{Col}(X,Y))\geq 1$;

\item[$\diamond$] il $X$ is homogeneous and $Y=R$ is the radial
vector field of $\mathbb{C}^n$, then 
$\mathrm{Col}(X,R)$ is an union of straight lines through the origin $0\in\mathbb{C}^n$. Il 
$X\wedge R\not=0$, then the vector fields $X$ and $R$ generate a singular foliation~$\mathcal{F}$ 
of dimension $2$ of $\mathbb{C}^n$. There is a holomorphic foliation 
$\widetilde{\mathcal{F}}$ on $\mathbb{P}^{n-1}_\mathbb{C}$ such
that $\mathcal{F}=\pi^*(\widetilde{\mathcal{F}})$. It is possible to 
prove that 
\[
\mathrm{Col}(X,R)=\pi^{-1}\big(\mathrm{Sing}(\widetilde{\mathcal{F}})\big)=\mathrm{Sing}(\mathcal{F}).
\]
\end{itemize}
\end{rems}

The various previous examples suggest the following conjecture: 

\begin{conjecture}\label{thm:hom}
{\sl Let $\mathrm{Ch}(X_1,X_2,\ldots,X_p)$
be a homogeneous $p$-chambar of degree
$\nu\geq 1$ on $\mathbb{C}^n$, where $p\geq 2$. 
Then $\mathrm{Col}(X_k,R)=\mathrm{Sing}(X_k)$\footnote{\,Recall that $\mathrm{Sing}(X_k)$ is the 
singular set of $X_k$:
\[
\mathrm{Sing}(X_k)=\big\{m\in\mathbb{C}^n\,\vert\,X_k(m)=0\big\}.
\]
} for any $k\geq 1$.
In particular $\dim\mathrm{Sing}(X_k)\geq 1$.}
\end{conjecture}

In the same spirit we have the following problem:

\begin{prob}\label{prob:ineq}
Let $\mathrm{Ch}(X_1,X_2,\ldots,X_p)$ 
be a $($non-homogeneous$)$ $p$-chambar such that 
$X_k(0)=~0$. Do the inequalities 
$\dim\mathrm{Sing}(X_k)\geq 1$ hold ?
\end{prob}

\begin{rem}
The problem is solved in the following cases: 
\begin{itemize}
\item[$\diamond$] $\nu=1$ (Theorem \ref{thm:linnil});

\item[$\diamond$] $p=2$ (Theorem \ref{thm:2sing});

\item[$\diamond$] rigid-chambars (Corollary \ref{cor:rigidp}).
\end{itemize}
\end{rem}

We proved the conjecture in the special case of 
homogeneous $3$-chambar on~$\mathbb{C}^2$ of 
degree $2$. In fact we will prove the following:

\begin{thm}\label{thm:hom3cham}
{\sl Let $\mathrm{Ch}(X_1,X_2,X_3)$ be a homogeneous
$3$-chambar on $\mathbb{C}^2$ of degree $2$. 
Then, after a change of variables, $X_j$ can be 
written as $a_jy^2\frac{\partial}{\partial x}$, 
where $a_1+a_2+a_3=0$. In particular, any 
homogeneous $3$-chambar on $\mathbb{C}^2$ of 
degree~$2$ is rigid.}
\end{thm}

Let $X$ be a homogeneous vector field of 
degree $d$ on $\mathbb{C}^2$. Then $X$ has 
$d+1$ invariant straight lines through 
$0\in\mathbb{C}^2$, counted with 
multiplicity. These lines are the solutions
of $f(x,y)=0$, where $f$ is the homogeneous 
polynomial of degree $d+1$ defined by 
\begin{equation}\label{eq:deff}
R\wedge X=f(x,y)\frac{\partial}{\partial x}\wedge\frac{\partial}{\partial y} 
\end{equation}
that is $f(x,y)=\det\left(
\begin{array}{cc}
x & y\\
X(x) & X(y)
\end{array}
\right)$.
We will assume that $f\not\equiv 0$ (if 
$f\equiv 0$, then $X$ is colinear to 
the radial vector field $R$).

Since $f=0$ is $X$-invariant, then $X(f)=h\cdot f$, 
where $h$ is a homogeneous polynomial of degree 
$d-1$. Moreover, $h=0$ if and only if $f$ is a 
first integral of $X$. In this case, the foliations
defined by $X$ and by~$f$ must coincide: the 
relation $X(f)=0$ gives 
$X(x)\frac{\partial f}{\partial x}+X(y)\frac{\partial f}{\partial y}$,
and thus 
$X(x)\frac{\partial f}{\partial x}=-X(y)\frac{\partial f}{\partial y}$.
Since the degrees of $X(x)$, $X(y)$, $\frac{\partial f}{\partial x}$,
and $\frac{\partial f}{\partial y}$ are equal, we obtain
that 
\[
X=\alpha\left(\frac{\partial f}{\partial x}\frac{\partial }{\partial y}-\frac{\partial f}{\partial y}\frac{\partial }{\partial x}\right).
\]
Using that $R(f)=(d+1)f$ and (\ref{eq:deff})
we get $\alpha=\frac{1}{d+1}$ in the above relation.

In general, we have 
\begin{equation}\label{eq:deff2}
(d+1)X-hR=H(f),    
\end{equation}
where $H(f)=\frac{\partial f}{\partial x}\frac{\partial }{\partial y}-\frac{\partial f}{\partial y}\frac{\partial }{\partial x}$.

Another relation that we will use is 
\begin{equation}\label{eq:deff3}
X(f)=X\left(\det\left(\begin{array}{cc}
x & y\\
X(x) & X(y)
\end{array}
\right)\right)=\det\left(\begin{array}{cc}
x & y\\
X^2(x) & X^2(y)
\end{array}\right).
\end{equation}

\begin{lem}\label{lem:1}
{\sl Let $\mathrm{Ch}(X_1,X_2,X_3)$ be a 
homogeneous $3$-chambar of
degree $d$ on~$\mathbb{C}^2$. For 
$1\leq j\leq 3$ define $f_j$ by 
$R\wedge X_j=f_j(x,y)\frac{\partial}{\partial x}\wedge\frac{\partial}{\partial y}$. Suppose that the 
$f_j$ are 
not identically $0$. Then 
\begin{itemize}
\item[$\diamond$] either $f_1$, $f_2$ and $f_3$
have two common linear factors,

\item[$\diamond$] or $f_j$ is a first integral
of $X_j$, $1\leq j\leq 3$.
\end{itemize}}
\end{lem}

\begin{proof}[{\sl Proof}]
First of all, using relations (\ref{eq:deff}), 
(\ref{eq:deff3}) and both
\begin{align*}
  &  \displaystyle\sum_jX_j(x)=\displaystyle\sum_jX_j(y)=0,
&&\displaystyle\sum_jX_j^2(x)=\displaystyle\sum_jX_j^2(y)=0
\end{align*}
we obtain $\displaystyle\sum_jf_j=0$ and 
$\displaystyle\sum_j X_j(f_j)=0$. If we set 
$X_j(f_j)=h_j\cdot f_j$, $1\leq j\leq 3$, then
$\displaystyle\sum_jh_j\cdot f_j=0$. On the 
other hand, since $\displaystyle\sum_jX_j=0$
and $\displaystyle\sum_jf_j=0$, we get from
(\ref{eq:deff2}) that
\[
0=\displaystyle\sum_j\big((d+1)X_j-h_jR-H(f_j)\big)=-\displaystyle\sum_jh_jR
\]
and so $\displaystyle\sum_jh_j=0$.

Let us assume that $h_j\not\equiv 0$ for some 
$1\leq j\leq 3$. In this case, from 
$\displaystyle\sum_jh_j=0$ there are 
$i\not=j$ such that $h_i\not=h_j$. Suppose 
for instance that $h_1\not=h_2$. Then 
the equalities 
\[
\left\{
\begin{array}{ll}
f_1+f_2+f_3=0\\
h_1f_1+h_2f_2+h_3f_3=0
\end{array}
\right.
\]
imply
\begin{equation}\label{eq:deff4}
(h_1-h_3)f_1=(h_3-h_2)f_2.
\end{equation}
In particular both members of relation 
(\ref{eq:deff4}) are not identically zero. 
Since $h_1-h_3$ and $h_3-h_2$ have degree
$d-1$, and $f_1$ and $f_2$ degree $d+1$, 
$f_1$ and $f_2$ must have two common factors. As 
$f_3=-f_1-f_2$ these factors are also 
factors of $f_3$.
\end{proof}

\begin{rem}
Lemma \ref{lem:1} implies that for a homogeneous 
$3$-chambar on $\mathbb{C}^2$ Problem \ref{prob:ineq}
has a positive answer, maybe except when the
$f_i$ are first integral.
\end{rem}

\begin{lem}\label{lem:2}
{\sl Let $\mathrm{Ch}(X_1,X_2,X_3)$ be a homogeneous
$3$-chambar of degree $2$ on $\mathbb{C}^2$, 
and let $f_\ell$ be as in Lemma~\ref{lem:1}.

Then the $f_\ell$'s are not identically zero.}
\end{lem}

\begin{proof}[{\sl Proof}]
Suppose that $f_1\equiv 0$; up to a linear 
change of coordinates we can assume that 
$X_1=xR=x^2\frac{\partial}{\partial x}+xy\frac{\partial}{\partial y}$. 
Let $\ell=0$ be a $X_2$-invariant line;
then $\ell=0$ is $X_1$-invariant, and also
$X_3$-invariant since $X_1+X_2+X_3=0$. These facts 
imply that the restriction of $X_1$, $X_2$, 
$X_3$ to $\ell=0$ define a $3$-chambar 
on the line $\ell=0$. The classificaiton 
of $3$-chambars on $\mathbb{C}$ 
(Theorem \ref{thm:loc3bfc}) implies that the $X_i$ are $0$
on $\ell=0$. In particular $\ell=0=(x=0)$ and 
$X_1=xR$, $X_2=xL_2$, $X_3=xL_3$, with 
$L_i$ linear vector field, and $R+L_2+L_3=0$. 
The same argument as before implies that 
the invariant lines of $L_2$, $L_3$ are 
necessarily $x=0$, {\it i.e.}:
\begin{align*}
    & L_2=a_2x\frac{\partial}{\partial x}+(b_2x+c_2y)\frac{\partial}{\partial y}, 
    && L_3=a_3x\frac{\partial}{\partial x}+(b_3x+c_3y)\frac{\partial}{\partial y}.
\end{align*}
The first components of the flows of $X_1$, $X_2$,
$X_3$ are respectively $\frac{x}{1-tx}$, 
$\frac{x}{1-ta_2x}$, $\frac{x}{1-ta_3x}$;
the sum of these three homographies can not be $3x$:
contradiction.
\end{proof}

\begin{prob}
Is Lemma \ref{lem:2} true in any degree ?
\end{prob}

Assume that $\mathrm{Ch}(X_1,X_2,X_3)$ 
is homogeneous of degree $2$, and that 
the $f_j'$s have two common factors. Let 
$\ell_1$ and $\ell_2$ be the two linear 
common factors of the $f_j'$s. We have
the following two possibilities:
\begin{enumerate}
\item[i)] $\ell_1\not=\ell_2$: we can thus 
assume that $xy$ is a factor of the $f_j'$s;

\item[ii)] $\ell_1=\ell_2$: we can thus suppose
that $y^2$ is a factor of the $f_j'$s.
\end{enumerate}
Another fact is that a polynomial $p$-chambar
in dimension $1$ is constant (Proposition 
\ref{pro:dim1pol}). Therefore, if a straight line
$\ell=0$ is invariant for all vector fields of 
the chambar, then $X_{j_{\vert_\ell}}=0$, and 
$\ell$ is a factor of $X_j$. In dimension $2$
this implies that $X_j=\ell\cdot L_j$, where 
$L_j$ is a linear vector field, $1\leq j\leq 3$.

In particular, i) and ii) imply the following 
possibilities:
\begin{enumerate}
\item[i')] if $\ell_1=x$ and $\ell_2=y$, then we 
must have $X_j=xyV_j$ where $V_j$ is a constant 
vector field;

\item[ii')] if $\ell_1=\ell_2=y$, then $X_j=yL_j$ where 
$R\wedge L_j=ym_j\frac{\partial}{\partial x}\wedge\frac{\partial}{\partial y}$, $m_j=a_jx+b_jy$. In particular, we must have
\begin{equation}\label{eq:deff5}
L_j=(\alpha_jx+\beta_jy)\frac{\partial}{\partial x}+\gamma_jy\frac{\partial}{\partial y}
\end{equation}
where $a_j=\gamma_j-\alpha_j$ and $b_j=-\beta_j$.
\end{enumerate}

Let us check that i') can not happen. In fact, 
let $X=xyV$, where 
$V=a\frac{\partial}{\partial x}+b\frac{\partial}{\partial y}$.
By a direct computation we find
\[
\left\{
\begin{array}{ll}
\frac{X(x)}{xy}=a,\quad\frac{X^2(x)}{xy}=a^2y+abx,\quad\frac{X^3(x)}{xy}=a^3y^2+\alpha xy+\beta x^2\\
\frac{X(y)}{xy}=b,\quad\frac{X^2(x)}{xy}=aby+b^2x,\quad\frac{X^3(y)}{xy}=b^3x^2+\gamma xy+\delta y^2
\end{array}
\right.
\]
This implies with obvious notations: for any $1\leq k\leq 3$
\begin{align*}
&a_1^k+a_2^k+a_3^k=0
&&\text{and} && b_1^k+b_2^k+b_3^k=0
\end{align*}
so $V_1=V_2=V_3=0$.

In situation ii') the vector fields $X_j=yL_j$ 
are of the form 
$X=y\left((ax+by)\frac{\partial}{\partial x}+cy\frac{\partial}{\partial y}\right)$, and a direct computation shows that 
$X(y)=cy^2$, $X^2(y)=2c^2y^3$, and $X^3(y)=6c^3y^4$.
This implies $\displaystyle\sum_jc_j^k=0$ for 
$1\leq k\leq 3$, so that $c_1=c_2=c_3=0$, and 
$X_j=y\ell_j\frac{\partial}{\partial x}$, where 
$\ell_j=a_jx+b_jy$ is linear. In particular, 
we get 
\begin{align*}
& X_j(x)=y\ell_j, && X^2_j(x)=y^2\frac{\partial\ell_j}{\partial x}\ell_j=a_jy^2\ell_j && X^3_j(x)=a_j^2y^3\ell_j; 
\end{align*}
as a consequence, $a_1^k+a_2^k+a_3^k=0$ for any $1\leq k\leq 3$.
This yields to $a_1=a_2=a_3=0$, and to 
$X_j=b_jy^2\frac{\partial}{\partial x}$ for any $1\leq j\leq 3$.
Note that the $f_j$ are first integral of $X_j$.

It remains to consider the case where $h_1=h_2=h_3=0$ and 
$f_j$ is a first integral of $X_j$, $1\leq j\leq 3$.
Let us come back to the definition of 
$f_j:=xX_j(y)-yX_j(x)$, so that $X_j(f_j)=0$. Remark 
first that $X$ is a constant multiple of the hamiltonian 
of $f$ 
\[
H(f_j)=\frac{\partial f_j}{\partial y}\frac{\partial}{\partial x}-\frac{\partial f_j}{\partial x}\frac{\partial}{\partial y};
\]
it can be checked that this follows from $X_j(f_j)=0$.
From the definition of $f_j$ and Euler's identity
we get $X_j=-\frac{1}{3}H(f_j)$. Let $\ell$ be a straight 
line invariant for $X_1$ and passing through $0$. 
Suppose by contradiction that it is not invariant by $X_2$.
We assert that, either $X_{1_{\vert\ell}}=0$, or the 
trajectories of $X_2$ and of $X_3$ are parallel straight 
lines. Assume that $X_{1_{\vert\ell}}\not=0$; we will 
see that $f_2$ is a perfect cube, {\it i.e.} $f_2=h^3$, 
where $h$ is linear, so that the trajectories of 
$X_2$ are the levels of $h$. Without lost of 
generality we can suppose that $\ell=(y=0)$. We can 
write $f_2(x,y)=ax^3+yq(x,y)$, where $q$ is homogeneous
of degree $2$ and $a\not=0$ because $y=0$ is not 
$X_2$-invariant. If $c\not=0$, then the level $f_2=c$
cuts $\ell$ in three points $z_j:=(x_j,0)$, $1\leq j\leq 3$,
where the $x_j$'s are the roots of $x^3=\frac{c}{a}$. 
If $f_2$ is not a perfect cube, then the level 
$f_2=c$ is irreducible, and so it is connected. 
Denote by $\varphi_t^j$ the flow of $X_j$, 
$1\leq j\leq 3$. Let us remark the following facts:
\begin{itemize}
\item[a)] 
$\varphi_t^1(x,0)+\varphi_t^2(x,0)+\varphi_t^3(x,0)=3(x,0)$
for all $x\in\mathbb{C}$, for all $t$ where the flows 
are defined (barycentric property);

\item[b)] 
$X_{1\vert_{y=0}}=\alpha x^2\frac{\partial}{\partial x}$ so
$\varphi_t^1(x,0)=\frac{x}{1-\alpha tx}$ and since we are 
assuming $X_{1\vert_{y=0}}\not=0$, $\alpha$ is non-zero;

\item[c)] as $(f_2=c)\cap(y=0)=\big\{(x_j,0)\,\vert\,1\leq j\leq 3\big\}$
and $f_2=c$ is connected, there exists $\tau\not=0$
such that $\varphi_\tau^2(x_1,0)=(x_2,0)$.
\end{itemize}
It is possible to prove that $\varphi_{3\tau}^2(x_1,0)=(x_1,0)$,
and more generally $\varphi_{3k\tau}^2(x_i,0)=(x_i,0)$ for all
$k\in\mathbb{Z}$, $i=1$, $2$, $3$. Since $f_3$ is a first integral
of $X_3$, the leaf of the foliation generated by $X_3$ through
$(x_1,0)$ must cut $\ell$ in at most three points. However, 
a) and b) imply that
\[
\varphi_{3k\tau}^3(x_1,0)=\left(2x_1-\frac{x_1}{1-3k\tau x_1},0\right),
\]
contradicting that the number is finite. As a result,
\begin{itemize}
\item[I)] either $f_2$ and $f_3$ are 
perfect cubes, 
\item[II)] or $X_{1\vert_{y=0}}=0$.
\end{itemize}

Let us deal with these two possibilities.

\begin{itemize}
    \item[I)] Assume that $f_2=\ell_2^3$ and $f_3=\ell_3^3$, 
where $\ell_2$ and $\ell_3$ are linear. In this case, 
the trajectories of $X_2$, and also of $X_3$, are
parallel lines. We have the alternative: 
\begin{itemize}
\item[A)] either $d\ell_2\wedge d\ell_3=0$, 
\item[B)] or $d\ell_2\wedge d\ell_3\not=0$.
\end{itemize}

In case A), we have $\ell_3=\alpha\ell_2$, $\alpha\not=0$, 
and $\ell_2$ is a line invariant for the chambar. After a 
linear change of variables we can suppose that 
$X_j=a_jy^2\frac{\partial}{\partial x}$, and the 
statement is proved. Note that in this case 
$X_{j_{\vert_\ell}}=0$ for $1\leq j\leq 3$.

In case B), after a linear change of variables, we 
can suppose that $f_2=-x^3$ and $f_3=-y^3$, which implies
$X_2=-x^2\frac{\partial}{\partial y}$, and 
$X_3=-y^2\frac{\partial}{\partial x}$. However, 
in this case we would have 
$X_1=y^2\frac{\partial}{\partial x}+x^2\frac{\partial}{\partial y}$. 
This is not a $3$-chambar because
\begin{align*}
& X_2^2(x)=X_3^2(x)=0 && \text{and} && X_1^2(x)\not=0.
\end{align*}

\item[II)] Suppose that $X_{1\vert_{y=0}}=0$. 
From the above we have the following consequences: the hamilto\-nian $H(f_j)=X_j$ is identically 
zero on the lines $f_j=0$. In particular all the irreducible components of~$f_j$ 
have multiplicity. Since the $f_j$'s have degree $3$, the $f_j$'s are perfect cubes and we conclude as previously.
\end{itemize}

This ends the proof of Theorem \ref{thm:hom3cham}.

\vspace{2cm}

\bibliographystyle{plain}

\bibliography{biblio}

\begin{thebibliography}{1}

\bibitem{CamachoSad}
C.~Camacho and P.~Sad.
\newblock Invariant varieties through singularities of holomorphic vector
  fields.
\newblock {\em Ann. of Math. (2)}, (115):579--585, 1982.

\bibitem{Cerveau}
D.~Cerveau.
\newblock Feuilletages en droites, \'{e}quations des eikonales et autres
  \'{e}quations diff\'{e}rentielles.
\newblock {\em Ast\'{e}risque}, (323):101--122, 2009.

\bibitem{CerveauDeserti}
D.~Cerveau and J.~D\'{e}serti.
\newblock {\em Transformations birationnelles de petit degr\'{e}}, volume~19 of
  {\em Cours Sp\'{e}cialis\'{e}s}.
\newblock Soci\'{e}t\'{e} Math\'{e}matique de France, Paris, 2013.

\bibitem{Humphreys}
J.~E. Humphreys.
\newblock {\em Introduction to {L}ie algebras and representation theory},
  volume~9 of {\em Graduate Texts in Mathematics}.
\newblock Springer-Verlag, New York-Berlin, 1978.
\newblock Second printing, revised.

\bibitem{PereiraPirio}
J.~V. Pereira and L.~Pirio.
\newblock {\em An invitation to web geometry}, volume~2 of {\em IMPA
  Monographs}.
\newblock Springer, Cham, 2015.

\bibitem{Tougeron}
J.-C. Tougeron.
\newblock {\em Id\'{e}aux de fonctions diff\'{e}rentiables}.
\newblock Springer-Verlag, Berlin-New York, 1972.
\newblock Ergebnisse der Mathematik und ihrer Grenzgebiete, Band 71.

\end{thebibliography}

\end{document}